\documentclass{amsart}
\usepackage{enumerate, amscd, amssymb, amsmath, mathtools}

\usepackage{hyperref}

\usepackage[english]{babel} 

\usepackage[utf8]{inputenc}

\allowdisplaybreaks
\usepackage{color}

\usepackage{tikz}
\usepackage[normalem]{ulem}
\renewcommand\sout{\bgroup\markoverwith
{\textcolor{red}{\rule[0.7ex]{3pt}{1.4pt}}}\ULon}

\newcommand\RE{\operatorname{Re}}

\makeatletter
\newcommand{\doublewidetilde}[1]{{%
  \mathpalette\double@widetilde{#1}%
}}
\newcommand{\double@widetilde}[2]{%
  \sbox\z@{$\m@th#1\widetilde{#2}$}%
  \ht\z@=.9\ht\z@
  \widetilde{\box\z@}%
}

\newcommand{\dist}{\operatorname{dist}}
\newcommand\seq{\, = \,}
\newcommand\define{\mathrel{\ := \ }}
\newcommand\ede{{\define}}

\newcommand{\CC}{\mathbb C}

\newcommand{\NN}{\mathbb N}

\newcommand{\RR}{\mathbb R}

\newcommand{\ZZ}{\mathbb Z}

\newcommand{\CI}{{\mathcal C}^{\infty}}
\newcommand{\CIc}{{\mathcal C}^{\infty}_{\text{c}}}

\newcommand\pa{\partial}

\newcommand{\maD}{\mathcal D}

\newcommand{\maL}{\mathcal L}

\newcommand{\maP}{\mathcal P}

\newcommand{\maU}{\mathcal U}
\newcommand{\maV}{\mathcal V}
\newcommand{\maW}{\mathcal W}
\newcommand{\maO}{\mathcal O}
\newcommand{\maX}{\mathcal X}
\newcommand{\maY}{\mathcal Y}

\newcommand\vect[1]{\mathbf{#1}}

\newcommand\nH[2]{|#2|_{\maV, #1}}
\newcommand\nW[2]{|||#2|||_{\maV, #1}}

\newcommand\nnW[2]{|||#2|||_{\maV, #1}}

\newcommand\scal[2]{ ( #1, #2 )}

\newtheorem{theorem}{Theorem}[section]
\newtheorem{proposition}[theorem]{Proposition}
\newtheorem{corollary}[theorem]{Corollary}

\newtheorem{lemma}[theorem]{Lemma}

\theoremstyle{definition}

\newtheorem{assumptions}[theorem]{Assumptions}
\newtheorem{definition}[theorem]{Definition}
\newtheorem{notation}[theorem]{Notation}
\newtheorem{remark}[theorem]{Remark}

\author[S. Labrunie]{Simon Labrunie} \address{
   Universit\'{e} de Lorraine, CNRS, IECL, F-54000 Nancy, France} 
\email{simon.labrunie@univ-lorraine.fr}

\author[H. Mohsen]{Hassan Mohsen} \address{
   Universit\'{e} de Lorraine, CNRS, IECL, F-54000 Nancy, France}
\email{hassan.mohsen@univ-lorraine.fr}

\author[V. Nistor]{Victor Nistor} \address{
   Universit\'{e} de Lorraine, CNRS, IECL, F-57000 Metz, France}

\email{victor.nistor@univ-lorraine.fr}

\begin{document}

\title[Parametric transmission problems]
{Polynomial bounds for the solutions of parametric transmission
problems on smooth, bounded domains}

\subjclass{35R01, 35J75, 46E35, 65N75}
\keywords{Parametric elliptic equations, strongly elliptic operators,
transmission problems, Poincar\'e inequality, Sobolev spaces, Finite Element Method,
uncertainty quantification}

\begin{abstract}
We consider a \emph{family} $(P_\omega)_{\omega \in \Omega}$ of elliptic second order 
differential operators on a domain $U_0 \subset \RR^m$ whose coefficients depend on the space variable 
$x \in U_0$ and on $\omega \in \Omega,$ a probability space. We allow the 
coefficients $a_{ij}$ of $P_\omega$ to have jumps over a fixed interface $\Gamma \subset U_0$ (independent 
of $\omega \in \Omega$). We obtain polynomial in the norms of the coefficients estimates on the norm of 
the solution $u_\omega$ to the equation $P_\omega u_\omega = f$ with transmission and mixed boundary conditions
(we consider ``sign-changing'' problems as well). 
In particular, we show that, if $f$ and the coefficients $a_{ij}$ are smooth enough and follow a 
log-normal-type distribution, then the map $\Omega \ni \omega \to \|u_\omega\|_{H^{k+1}(U_0)}$ 
is in $L^p(\Omega)$, for all $1 \le p < \infty$. The same is true for the norms of the inverses
of the resulting operators. We expect our estimates to be useful in Uncertainty Quantification.
\end{abstract}

\maketitle

\tableofcontents

\section{Introduction}

\subsection{Setting and statement of main result}
We consider the second order transmission/mixed boundary value problem
\begin{align}\label{eq.classical}
  \left\{ \begin{array}{rcl}
  P^{A} u \ = \  f & \text{ in } & U_0 \\ 
  u \ = \ \tilde g & \text{ on } & \partial_D U_0  \\
  \partial_{\nu}^{A} u \ = \ g & \text{ on } & \partial_N U_0  \\
  \, [[u]]_{\Gamma} \ = \ \tilde h & \text{ on } & \Gamma \\
  \, [[\partial_{\nu}^{A} u]]_{\Gamma} \ = \ h & \text{ on } & \Gamma 
\end{array}\right . 
\end{align}
on a \emph{bounded (open) domain} $U_0 \subset \RR^m$. The symbols appearing in
this problem have the following meaning. The matrix $A = [a_{ij}] \in M_{m+1}(L^\infty(U_0))$, 
$i, j = 0, \ldots, m$, is the matrix of coefficients of
\begin{equation}
   P^{A} u(x) \ede \sum_{i,j = 0}^m \pa_i \big ( a_{ij}(x) \pa_j u(x) \big )\,,
\end{equation}
where we have set $\pa_0 := id$, for convenience, and $\pa_j := \frac{\pa}{\pa x_j}$ for $j \ge 1$,
as usual. We assume that the domain $U_0$ on which our problem is formulated is open and bounded and is 
\emph{decomposed} into open subsets $U_j$, meaning that  
\begin{equation}\label{eq.decomp.U0}
   \overline{U}_0 \seq  \bigcup_{j=1}^N \overline{U}_j
\end{equation} 
where $U_j \subset U_0$, $j = 1, \ldots, N$, are open and disjoint. We allow the 
coefficients $a_{ij}$ of $P^{A}$ to have jumps over a fixed interface $$\Gamma := \bigcup_{j=1}^N \pa U_j 
\smallsetminus \pa U_0 \subset U_0 \,.$$  We assume that $\pa U_0
= \pa_D U_0 \cup \pa_N U_0$ is a disjoint partition with $\pa_D U_0$ and $\pa_N U_0$ both 
closed (and hence, both open in $\pa U_0$). This is necessary in order to avoid singularities
from the jump in the type of boundary conditions. The treatement of jumps in boundary conditions 
would require additional techniques that would render the paper unreadable. 

Our main result is a \emph{polynomial estimate} in the norm of the coefficients $A$ for the norm 
of the solution $u$ of the problem~\eqref{eq.classical}, see Theorem~\ref{thm.transmission} shortly
below. To formulate this theorem, we shall need some notation. Let 
\begin{equation}
   H_D^1(U_0) \ede \{ v \in H^1(U_0) \mid v\vert_{\pa_D U_0} = 0\}
\end{equation}
and 
\begin{equation}
   \maP_0 \seq \maP_0^{A} : H_D^1(U_0) \to H_D^1(U_0)^*
\end{equation}   
be the operator of the weak formulation of our problem (see Definition~\ref{def.opsP.k}, for 
the standard definition). Let 
\begin{equation}
   \check H^k(U_0) \ede \bigoplus_{j=1}^N H^k(U_j) \quad \mbox{ and } \quad \check W^{k,\infty}(U_0) 
   \ede \bigoplus_{j=1}^N W^{k,\infty}(U_j)
\end{equation} 
be the \emph{broken Sobolev spaces} associated to the decomposition of $U_0 \subset \RR^m$ into the disjoint subdomains 
$U_j$. It is convenient to set $F \ede (f, \tilde g, g, \tilde h, h)$ for the data in Problem \eqref{eq.classical} and 
\begin{equation}\label{eq.def.normF}
   \begin{gathered}
      \|F\|_{V_{k}^-} \ede \| f \|_{\check H^{j}} + \|\tilde g\|_{H^{j+ \frac32}} +
       \|\tilde h\|_{H^{j+\frac32}} + \|g\|_{H^{j+\frac12}} + \|h\|_{H^{j+\frac12}} \ \mbox{ if}\\
       F \in \check H^{k}(U_0) \oplus \check H^{k+ 3/2}(\pa_D U_0) 
       \oplus \check H^{k+ 1/2}(\pa_N U_0) \oplus \check H^{k + 3/2}(\Gamma) \oplus \check H^{k+ 1/2}(\Gamma)\,.
   \end{gathered}
\end{equation}
If $E$ is a normed space, we shall typically write $\| \cdot \|_E$ for its norm. If $\xi \notin E$,
we let $\|\xi \|_{E} := \infty$. We let $\maU := \{U_j \mid j = 0, \ldots, N\}$ be our family of open
subsets of $\RR^m$ as in Equation \eqref{eq.decomp.U0}.

\begin{theorem}\label{theorem.mainI} Let us assume that $U_0$ is bounded, 
   that all $\pa U_j$, $j = 0, \ldots, N$, are smooth, and that $\pa_D U_0$, $\pa_N U_0$,
   and $\Gamma$ are closed and disjoint. Then there exist $C_{k,\maU} > 0$,
   $k \ge 0$, with the following property. Let 
   $A = [a_{ij}]\in M_{m+1}(\check W^{k+1,\infty}(U_0))$ be such   
   that $\maP_0 := \maP_0^A : H_D^1 (U_0) \to H_D^1(U_0)^*$ is invertible. Then, for every 
   $F := (f, \tilde g, g, \tilde h, h)$,
   as in Equation \eqref{eq.def.normF}, we have
\begin{equation*}
     \| u \| _{\check H^{k+2}(U_0)} \leq C \,
      \sum_{q=0}^{k+1} \, 
      \|  \maP_0^{-1} \|^{q+1} \,
      \nnW{k}{a_{mm}^{-1}}^{(q+1)k+1} 
      \|A\|_{\check W^{k+1,\infty}(U_0)}^{(q+1)(k+1)} 
      \|F\|_{V_{k+1-q}^-}\,.
\end{equation*} 
   where  $\nnW{k}{a_{mm}^{-1}}^{(q+1)k+1}$ is the $W^{k,\infty}$ norm of the coefficients $(a_{mm}^{\alpha})^{-1}$
   of the matrix $A$ in suitable adapted coordinates $\phi_\alpha : W_\alpha \to W_\alpha'$
   near the boundary (see Definition~\ref{def.amm}).
\end{theorem}

The bound $C_{k, \maU}$ may depend, as the notation shows, on $k$ and the open subsets
$\maU := \{ U_0, U_1, \ldots, U_N\}$ of Equation \eqref{eq.decomp.U0}.
Of course, the main point of our result is that the bound $C_{k, \maU}$ is independent of the coefficients
$A$ of the operator. This is relevant since, for a fixed operator, the result was known in the strongly elliptic 
case (it is a very classical result if there is no interface). Another important feature of our 
result is that it applies to sign changing problems (the non-definite case)
for instance, by using ``$T$-coercivity'' \cite{Ciarlet12, Ciarlet13}. It applies also to self-adjoint
operators, as discussed in Section~\ref{sec.self-adjoint} (the last section).

\subsection{The probabilistic setting}
Let us assume that the coefficients $a_{ij}$ depend on both the space variable $x \in U_0$ and on an 
additional variable $\omega \in \Omega$ in a probability space. This is a setting that appears in
practice, because the coefficients $a_{ij}$ represent properties of materials that are not always
known exactly. Our results then translate into estimates for the family $P^{A(\omega)}$ and the associated 
transmission/mixed boundary value problem. One of the reasons Theorem~\ref{theorem.mainI} is important is
that it can be used to treat log-normal-type distributions for the coefficients. 
Indeed, we show that, if $f$ and the coefficients $a_{ij}$ 
are smooth enough and follow a log-normal-type distribution, then the map $\Omega \ni \omega \to 
\|u_\omega\|_{H^{k+1}(U_0)}$ is in $L^p(\Omega)$, for all $0 \le p < \infty$.
The same is true for the norm of $P_\omega^{-1}$ 
(acting between suitable spaces that take into account the boundary and transmission conditions),
see Theorems \ref{thm.integr} and~\ref{thm.integr2}.
We apply our results to obtain Finite Element estimates for Gaussian families and to prove that the
error is integrable, Theorem~\ref{thm.integr3}, thus yielding, in a certain sense, optimal rates
of convergence. We expect these estimates to be useful in Uncertainty 
Quantification~\cite{Schwab1, SchwabBookUQ, LeM1, SullivanBookUQ}.

\subsection{Motivation: sign changing problems and connection with boundary triples}
The initial motivation for this paper (which is largely based on the Ph. D. Thesis of the
second named author~\cite{theseHassan}) is to study the regularity of problems with ``sign-changing'' 
coefficients for \emph{polygonal domains}. Here ``sign changing'' means that the matrix $A$ can be 
definite of different signs on different 
subdomains $U_j \subset U_0$.
(For instance, our operator could be of the form $P^A = c_j \Delta$ with, $c_j$ constant on 
$U_j$, with these constants $c_j$ possibly of different signs.) An approach to the study of these
problems is via ``T-coercivity,'' as in the paper by Ciarlet, Chesnel, and Bonnet-Ben Dia 
\cite{Ciarlet12}. See also~\cite{Ciarlet14, Karim2, Renata, Ciarlet13, Pankrashkin19} for further 
motivation and results on this type of problems. When studying regularity on polygonal domains, we are
facing several issues, among which is the development in singular functions and regularity
in weighted (i.e. Kondratiev) spaces. As is well known, see, for instance~\cite{daugeBook, Kondratiev67,
KMR, NicaiseBook}, it is convenient to study a certain induced operator on the circle. More precisely,
we are interested in whether it is self-adjoint and in its spectrum, since the spectrum of
this self-adjoint operator gives the singular values and singular functions of the operator. 
A numerical study of this spectrum is contained in~\cite{theseHassan}. The self-adjointness
problem also arises in the study of the regularity--even for smooth domains. These problems
are conveniently approached in our setting using the techniques of boundary triples, as
in the paper by Br\"uning, Geyler, and Pankrashkin~\cite{Bruning}. Results in this
direction were obtained in~\cite{theseHassan}. Using self-adjointness results, we then prove 
a parametric regularity result for problems with sign changing coefficients, Theorem~\ref{thm.param.reg}.

\subsection{Contents} The paper is organized as follows. In Section~\ref{sec2} we introduce our 
setting: the domain $U_0 \subset \RR^m$ and its decomposition in subdomains, the basic 
spaces (the broken Sobolev spaces $\check H^k$ and $\check W^{k,\infty}$, $V_k := \check H^{k+1}(U_0) \cap H_D^1(U_0)$,
for $k \ge 0$ and $V_k^- := \check H^{k-1}(U_0) \oplus H^{k-1/2}(\Gamma \cup \pa_N U_0)$, for
$k \ge 1$) and the operators $\maP_k^A : V_k \to V_k^-$, $\maP_k^A u = (P^A u, D_\nu^A u, [[D_\nu^A u]])$.
We also formulate Nirenberg's trick, which is used in ``raising the regularity'' in the induction 
step proof of our main Theorem~\ref{thm.transmission}. In Section~\ref{sec3} we localize
our spaces and differential operators in order to be able to use the approach 
of~\cite{Victor, theseHassan, HassanSimonVictor1}. Essentially, we ``straighten the boundary or
the interface,'' but in a more sophisticated way. More precisely, we introduce
local bases $(X_j^\alpha)$, $j = 1, \dots, m$, of vector fields in local coordinate patches $W_\alpha$ 
and prove the needed technical results in terms of these vector 
fields. Section~\ref{sec4} contains the main estimates of this paper. It also contains some immediate
applications, including the boundedness (and hence integrability) result for the norms of $(\maP_k^A)^{-1}$ 
if both $A$ and $(\maP_0^A)^{-1}$ are bounded with respect to the parameters. In Section~\ref{sec5}, we extend 
this integrability result to some situations when the norms of both $A$ and $(\maP_0^A)^{-1}$ are not 
bounded but are rather derived from some Gaussian variables. The operators of this section are assumed 
to be strongly elliptic. The last section includes some
applications to the case when the operators are not strongly elliptic (problems with ``sign changing 
coefficients'').

\subsection{Acknowledgements}
We thank Patrick Ciarlet, Iulian C\^impean, Mirela Kohr, and Konstantin Pankrashkin for useful discussions.

\section{The domain, the operator(s), and Nirenberg's trick}
\label{sec2}

In this section we formulate more precisely our mixed boundary value/transmission problem
\eqref{eq.classical} and remind some useful results. We also introduce our domain $U_0$ and
its decomposition into subdomains $U_j$ and the basic spaces and
operators. \emph{Throughout this paper, $U_0$ will be
an open subset of $\RR^m$ with smooth boundary. For the main results,
we shall assume that $U_0$ has compact closure.}

We begin by introducing our bilinear form $B^A$ and one of the basic differential
operators $\maP^A$ and $P^A$ that we will use. The operators $\maP_k^A$ to which our 
main results apply are obtained from $\maP^A$.
By $H^k(U_0)$ we denote the usual Sobolev spaces and by $L^p(U_0)$ the usual space of  
$p$-integrable functions~\cite{BrezisBookFA, Evans, Hebey1, LionsMagenes1, TaylorPDEI}.

\subsection{The sesquilinear form $B^A$ and the operator $\maP^A$}
\label{ssec.operators}
Our main objects of study (forms, operators, ... ) are based on the
coefficient matrix
\begin{equation}\label{eq.def.A}
   A \ede [a_{ij}] \in M_{m+1}(L^\infty(U_0)) \,.
\end{equation}
We let $\partial_j := \frac{\pa}{\partial x_j}$, for 
$j = 1, 2, \ldots, m$, as usual, and $\partial_0 := id$. 

\begin{definition}\label{def.eq.def.forme.B}
Let $A$ be the $(m+1) \times (m+1)$ matrix of Equation
\eqref{eq.def.A}. We let $B = B^{A} : H^1(U_0) \times H^1(U_0) \to \CC$ 
be the sesquilinear form
\begin{equation*}
\begin{gathered}
   B(u, v) \seq B^A(u, v)  \ede  \int_{U_0} \, \sum_{i, j=0}^m 
   a_{ij}(x) \pa_j u(x) \pa_i \overline{v(x)}  
   \, dx\\
   \seq \int_{U_0} \, \Big (
   \sum_{i, j=1}^m a_{ij} \pa_j u \pa_i \overline{v} + 
   \sum_{j=1}^m a_{0j} (\pa_j u)  \overline{v} +
   \sum_{i=1}^{m} a_{i0} u  \pa_i \overline{v} + a_{00} u \overline{v} \Big )  \, dx\,.
\end{gathered}
\end{equation*}
\end{definition}

Our operator $\maP \seq \maP^A : H^1(\Omega) \to H^1(\Omega)^*$
and its restriction $P \seq P^A : H^1(\Omega) \to H^{-1}(\Omega)$
will be obtained from the sesquilinear form $B = B^A$ through a weak formulation. 
To that end, we need to recall the following conventions and notation.
Let $V^*$ be the \emph{conjugate linear} dual of some topological vector
space $V$ (in our applications, $V$ will always be a Hilbert space).
Let $$\langle \cdot , \cdot \rangle_{V^*, V} : V^* \times V \to \CC$$ 
be the duality between $V$ and $V^*$, which is, hence, a sesquilinear
form (just like $B = B^A$). We shall use this duality mostly for
$V = H_0^1(U_0)$, for which $V^* = H^{-1}(U_0) := H_0^{1}(U_0)^*$.
We can now introduce the two operators $\maP^A$ and $P^A$ associated to 
$A \in M_{m+1}(L^\infty(\Omega))$. Recall that $\partial_j := \frac{\pa}{\partial x_j}$, 
for $j = 1, 2, \ldots, m$, and that $\partial_0 := id$.

\begin{definition} \label{def.opsP} 
Let $ A \ede [a_{ij}] \in M_{m+1}(L^\infty(U_0))$ and $B^A$ be as in Definition
\ref{def.eq.def.forme.B}. 
We define the \emph{full operator} $\maP = \maP^A : H^1(U_0) \to H^{-1}(U_0)^*$ by
\begin{equation*}
 \langle \maP^A u, \phi \rangle_{H^1(U_0)^*, H^1(U_0)} \seq B^A(u, \phi) \,.
\end{equation*}
Similarly, the \emph{``plain'' differential operator} $P = P^A : H^1(U_0) \to H^{-1}(U_0)$ is
\begin{equation*} 
    Pu \seq P^{A}u  \ede -\sum_{j=1}^m \sum_{i=1}^m  \pa_i  
    (a_{ij}  \partial_j u ) - \sum_{j=1}^{m} \pa_j ( a_{j0} u)  
    + \sum_{i=1}^m a_{0 i} \pa_i u + a_{00} u \,.
\end{equation*}
\end{definition}

The definition of $\maP = \maP^A$ justifies the use of the \emph{conjugate} linear
dual. In turn, this is required by the fact that $B^A$ is sesquilinear, which, in 
turn, is convenient when using positivity (as in the Riesz or Lax-Milgram Lemma's). 
Let us continue with some remarks, the first one on the notation.

\begin{remark}\label{rem.d0}
The choice of the operator $\partial_0 := id$, the identity operator, 
is as in~\cite{Mirela1, theseHassan, HassanSimonVictor1}, and is convenient for
giving a compact form for the form $B^A$:
\begin{equation*}
   B^A(u, v) \ede \sum_{i,j = 0}^m (a_{ij} \pa_j u , \pa_i v)\,,
\end{equation*}
where $(f, g) := \int_\Omega f(x) \overline{g(x)} \, dx$ is the scalar product
on $L^2(\Omega)$. Then $\pa_j^* = - \pa_j$ for $j > 0$,
but (obviously), $\pa_0^* = \pa_0$. Thus $P^A = \sum_{i, j = 0}^m \pa_i^* a_{ij} \pa_j$.
Let $\CIc(W)$ denote the set of smooth, compactly supported functions on some open
set $W$ of a smooth manifold.
Note that the map $P^A : \CIc(U_0) \to \CIc(U_0)$ does not determine the coefficient
matrix $A$ (we can commute $\pa_i \pa_j = \pa_j \pa_i$ or move derivatives past
coefficients). 
\end{remark}

The second remark concerns the important relation between $\maP^A$ and $P^A$.

\begin{remark}\label{rem.dual}
For any $\phi \in H_0^1(U_0)$ and $u \in H^1(U_0)$,
the relations defining $\maP^A$ and $P^A$ (Definition~\ref{def.opsP}) yield 
\begin{equation*}
    \langle \maP^A u , \phi\rangle \ede B^A(\phi, u) \seq (P^A u, \phi)\,.
\end{equation*}
The inclusion $j : H_0^1(U_0) \to H^1(U_0)$ yields by duality a 
surjective restriction map $j^* : H^1(U_0)^* \to H^{-1}(U_0)$. We thus obtain
that
\begin{equation}\label{eq.rel.j}
    P^A \seq j^* \circ \maP^A\,.
\end{equation}
The operator $\maP^A$ thus contains more information than the differential 
operator $P^A$. We will see later that the additional information involves
boundary and transmission conditions. In fact, our main interest is in some closely related 
operators $\maP_k \seq \maP_k^A : V_k  \to V_{k}^-,$ $k \ge 0.$ The operator $\maP_k^A$ is a 
restriction of $\maP_0^A$ and,
in turn, this operator, is somewhere ``between'' $\maP^A$ and $P^A$, satisfying a relation
similar to Equation~\eqref{eq.rel.j}. The relation
between the operators $P^A, \maP^A$, and $\maP_k^A$, $k \ge 0$, is discussed in 
Remarks \ref{rem.relation0} and~\ref{rem.relation1}.
\end{remark}

The matrix $A = [a_{ij}] \in M_{m+1}(L^\infty(U_0))$
is typically fixed, so it will be often omitted from the notation.
(Anticipating, in the case of domains of manifolds, we will need to replace the matrix
$A$ with an endomorphism of the bundle $\CC \oplus T^*M$, as in~\cite{KohrNistor1}. This
will be treated in a forthcoming paper.)

\subsection{Decomposition of the domain $U_0$}
\label{ssec.2.2}
In applications, we shall assume that the coefficients $a_{ij} = a_{i,j} \in L^\infty(\Omega)$, 
$i, j = 0, \ldots, m$, of the matrix $A = [a_{ij}]  \in M_{m+1}(L^\infty(U_0))$ are \emph{piecewise regular.} To defined exactly
our assumptions on the coefficients $a_{ij}$, we need to specify more on the geometry of our
domain $U_0$, namely, to introduce a decomposition of this domain and the resulting interface
$\Gamma$. This is the purpose of this subsection.
Recall that, throughout this paper, $U_0 \subset \RR^m$ will be an open subset. For simplicity, 
we shall also assume that $U_0$ is connected. This is no loss of generality, since the 
general result can be obtained by studying one connected component at a time. 

We will make the following assumptions (see~\cite{HassanSimonVictor1} for a 
related, but different setting):

\begin{assumptions} \label{assum.decomposition}
For any subset $W$ of a topological space, we shall 
let $\pa W$ denote the topological
boundary of $W$. (Thus $\pa W := \overline{W} \smallsetminus W$ if $W$ is open.)
\begin{enumerate}
 \item We are given $N$ \emph{disjoint, open} subsets $U_j \subset U_0$, $j=1, \ldots, N$, such that
\begin{equation*} 
   \overline{U}_0 \seq \overline{U}_1 \cup \overline{U}_2 \cup \ldots \cup 
   \overline{U}_N\,.
\end{equation*} 

  \item Each $U_j$, $j = 0, \ldots, N$ is connected with $\partial U_j$ smooth.

  \item We are given a vector field $\nu^{[j]}$ that is smooth on
  $\overline{U}_j$ and that, on $\pa U_j$, is of unit lenght and outside pointing 
  normal to $\pa U_j$.
  
  \item The \emph{interface} $$\Gamma \ede \left( \bigcup_{j=1}^N \partial U_j \right) \smallsetminus \partial U_0$$
  is a compact, smooth submanifold of $\RR^m$ (without boundary).
  
  \item We have fixed a vector field $\vect{N} : \Gamma \to \RR^m$, unit normal to
   $\Gamma$ at every point. (This is the same thing as choosing an orientation of $\Gamma$.)
  
  \item We are given $\partial_N U_0 \subset \partial U_0$ closed and open in $\pa U_0$ and we 
  set
  \begin{equation*}
      \partial_D U_0 \ede \pa U_0 \smallsetminus  \partial_N U_0 \,,
  \end{equation*}
  (which will hence also be closed).
  
  \item We will let $\nu: U_0 \smallsetminus \Gamma \to \RR^m$ denote the vector field
that on $U_j$ coincides with $\nu^{[j]}$.
\end{enumerate}
\end{assumptions}

See Figure~\ref{image3} for an illustration of our domains. Note that $\pa_D U_0$,
$\pa_N U_0$, and $\Gamma$ are all compact, smooth submanifolds of $\RR^m$ and that
they are \emph{disjoint}. The fact that they are disjoint will prevent thus the 
appearence of singularities at the places where these sets would meet. 

\begin{center}
\begin{figure}[!h]
\includegraphics[height=5cm]{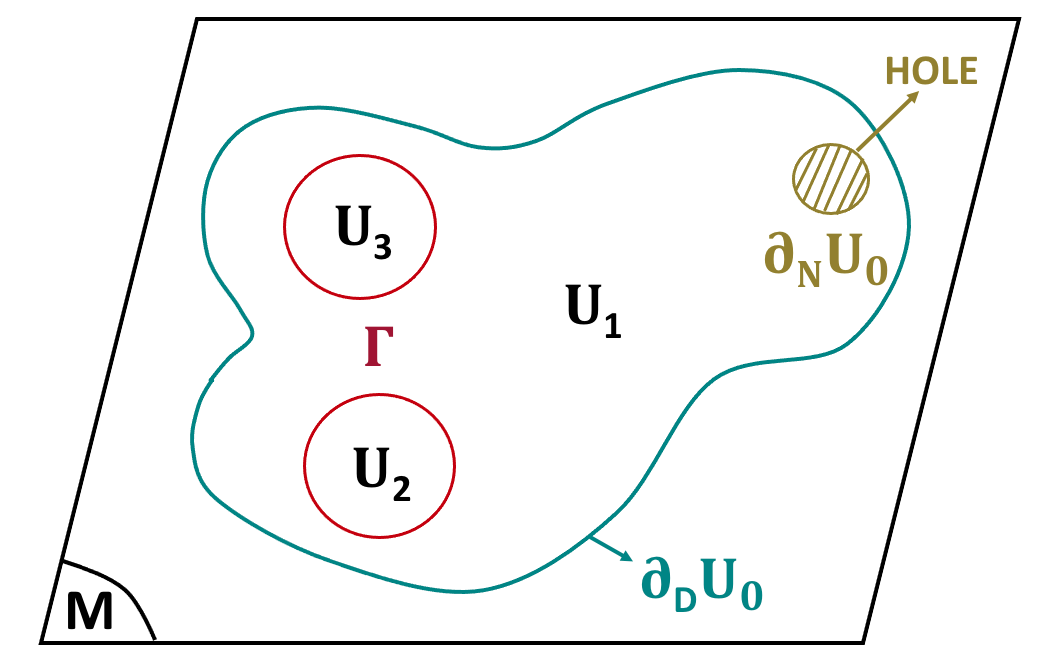}
\caption{\normalsize An illustration of the domains $U_j \subset M$ 
and of the interface $\Gamma$.}
\label{image3}
\end{figure}
\end{center}

Note that the boundary $\partial U_j$ of $U_j$ is not
assumed to be connected. We will impose Neumann type boundary conditions on
$\partial_N U_0$ and Dirichlet type boundary conditions on $\partial_D U_0$.
Also, as $U_0$ is open, we have
\begin{equation}\label{eq.assumpt.2}
   U_0  \seq U_1 \cup U_2 \cup \ldots \cup U _N \cup \Gamma\,.
\end{equation}

\subsection{Function spaces: broken Sobolev spaces, jumps, and traces} 
\label{ssec.2.3}
Our hypotheses ($\Gamma$ and all $\pa U_j$ smooth) imply that
$\Gamma$ is closed and that $\Gamma \cap \partial U_0 = \emptyset$. 
In this framework, it is natural to consider the so-called ``broken Sobolev spaces,''
as in other papers~\cite{NicaiseBook, LiNistorQiao}.

\begin{definition}\label{def.broken.SS}
Let $U_0 \subset \RR^m$ be an open subset as in Assumptions~\ref{assum.decomposition} and 
$\maO \subset U_0$ be an open subset. We let
\begin{equation}\label{eq.def.esb}
   \begin{gathered}
   \check{H}^s(\maO) \ede \bigoplus_{j=1}^N H^s(\maO \cap U_j) \,, \\
   \check{W}^{s, \infty}(\maO) \ede \bigoplus_{j=1}^N W^{s, \infty}(\maO \cap U_j)
   \,, \quad s \ge 0\,,
   \end{gathered}
\end{equation}
with the standard direct sum norms denoted $\|\cdot \|_{\check H^s(\maO)}$ and $\|\cdot \|_{\check W^{s, \infty}(\maO)}$.
\end{definition}

We will forget the open set $\maO$ in the notation of the norms on the above spaces
if there is no danger of confusion (for instance, if $\maO = U_0$), in particular, the norms 
on these spaces will be generically denoted $\|\cdot \|_{\check H^s}$ and $\|\cdot \|_{\check W^{s, \infty}}$
instead $\|\cdot \|_{\check H^s(\maO)}$ and $\|\cdot \|_{\check W^{s, \infty}(\maO)}$.
All our function spaces will be complex vector spaces. The usual trace constructions yield
the ``jumps'' on the interface, whose precise definitions depend on our choices of orientations,
as explained next.

\begin{remark}\label{rem.side-traces}
Let $\Gamma_k \subset \pa U_j$ be a connected component of $\Gamma$.
\begin{enumerate}[(i)]
\item First, it follows from definitions that, for all $p \in \Gamma_k \subset \overline{U})j$, 
we have $N(p) = \pm \nu_j(p)$. If we have $\vect N(p) = + \nu^{[j]}(p)$, we will say that $U_j$ 
is {\em to the right} of $\Gamma$ in $p$. 
Otherwise, we will say that $U_j$ is {\em to the left} of $\Gamma$ in $p$. 

\item We have exactly one domain $U_{j+} = U_+$ to the right of (every point of) $\Gamma_k$ 
and exactly one domain $U_{j-} = U_-$ to the left of (every point of) $\Gamma_k$, since the
domains $U_j$ are assumed to be connected.

\item If $u \in \check H^1(U_0)$, we will 
then let $$u^+ := u \vert_{U_+} \in H^1(U_+) \ \mbox{ and } \
u^- := u \vert_{U_-} \in H^1(U_-) \,,$$ which will therefore 
have traces (restrictions) on the boundaries $\partial U_{\pm}$.

\item Therefore we also obtain (two side) \emph{traces} 
$$u^{\pm}\vert_{\Gamma_k} \in H^{1/2}(\Gamma_k)$$ 
of $u$ on $\Gamma_k = \partial U_+ \cap U_-$.
\end{enumerate}
\end{remark}

We are now ready to introduce the jump over the interface $\Gamma$.

\begin{definition}\label{def.saut}
Using the notation of Remark~\ref{rem.side-traces}, we define the jump 
$[[u]]_\Gamma \in H^{1/2}(\Gamma) = \bigoplus_k H^{1/2}(\Gamma_k)$ 
of $u \in \check H^1(U_0)$ (where $\Gamma_k$ are 
the connected components of the interface $\Gamma$) by
\begin{equation*} 
\begin{gathered}
  {[[u]]}_{\vert \Gamma_k} \ede u^+\vert_{\Gamma_k} - u^-\vert_{\Gamma_k} 
  \in H^{1/2}(\Gamma_k) \quad \mbox{ and} \\
  [[u]]_\Gamma := [[u]] _{\vert \Gamma_k} \mbox{ sur } \Gamma_k \subset \Gamma \, .
\end{gathered}
\end{equation*}
\end{definition}

We have the following simple remark

\begin{remark} \label{rem.surj.tr}
Let $u \in \check H^1(U_0)$. We have $u \in H^1(U_0)$ if, and only if, $[[u]]_{\Gamma} = 0$. 
Moreover, the map $[[ \, \cdot \, ]]_{\Gamma} : \check H^1(U_0) \to H^{1/2}(\Gamma)$ is 
surjective. That is, we have an exact sequence
\begin{equation*}
    0 \longrightarrow H^1(U_0) \longrightarrow \check H^1(U_0) 
    \stackrel{[[ \, \cdot \, ]]_{\Gamma}}{ - \! \! \! - \! \! \! \longrightarrow} 
    H^{1/2}(\Gamma) \longrightarrow 0 \,.
\end{equation*}
In particular, $\check H^1(U_0) \simeq H^1(U_0) \oplus H^{1/2}(\Gamma)$, but this isomorphism 
is not canonical (it depends on suitable choices). See~\cite{theseHassan} for a proof.
\end{remark}

For us, a space more important than $H^1(U_0)$ will be the space
\begin{equation*}
     H_{D}^{1}(U_0) \ede \{ u \in H^1(U_0) \mid u \vert_{\partial_D U_0} = 0\}\,.
\end{equation*}

\begin{remark}
By analogy with Remark~\ref{rem.surj.tr}, we have an exact sequence:
\begin{equation*}
    0 \longrightarrow H_D^1(U_0) \longrightarrow \check H^1(U_0) 
    \cap \{ w\vert_{\partial_D U_0} = 0 \}
    \stackrel{[[ \, \cdot \, ]]_{\Gamma}}{ - \! \! \! - \! \! \! \longrightarrow} 
    H^{1/2}(\Gamma) \longrightarrow 0 \,.
\end{equation*}
In particular, $\check H^1(U_0) \cap \{ w\vert_{\partial_D U_0} = 0 \} 
\simeq H_D^1(U_0) \oplus H^{1/2}(\Gamma)$ non-canonically.
\end{remark}

Below, we shall use without further comment the following identifications.

\begin{remark}\label{rem.L2}
$\check H^0(U_0) = H^0(U_0) = L^2(U_0)$ and 
$ \check W^{0, \infty}(U_0) = W^{0, \infty}(U_0) = L^\infty(U_0) \,.$
\end{remark}

\subsection{The spaces $V_k$ and $V_k^-$ and the operators $\maP_k^A$}
\label{ssec.Pk}
As in~\cite{Victor, Mirela1, theseHassan, HassanSimonVictor1} and in other
papers, in order to study weak solutions, it is convenient to consider the spaces
$V_k$ and $V_k^-$ introduced next:

\begin{definition} \label{def.broken}
Let $U_0 \subset \RR^m$ satisfy Assumptions~\ref{assum.decomposition}.
For $k \ge 0$, we let:
\begin{align*}
  V_{k}\, & \ede \check{H}^{k+1}(U_0) \cap H_{D}^{1}(U_0) \\
  V_{k}^{-}  & \ede \check H^{k-1}(U_0) 
  \oplus  H^{k-\frac{1}{2}}(\partial_N U_0 \cup \Gamma) \,, \quad \mbox{if } k > 0\,, \mbox{ and}\\
  V_{0}^{-} & \ede V_0^* \coloneqq H_{D}^{1}(U_0)^* \, .
\end{align*}
\end{definition}

\begin{remark} \label{rem.inclusion}
Explicitly, for $k \ge 0$, we have
\begin{equation*}
  V_{k}\, \seq \{ w \in \bigoplus_{i=1}^{N} H^{k+1}(U_i) \mid [[w]]_\Gamma = 0\,, 
  \ w\vert_{\partial_D U_0} = 0 \}  \,.
\end{equation*}
In particular, $V_0 = H_D^1(U_0)$. Similarly, for $k > 0$, we have that
\begin{equation*} 
  V_{k}^{-}  \seq \bigoplus_{i=1}^{N}  H^{k-1}(U_i) 
  \oplus  H^{k-\frac{1}{2}}(\partial_N U_0 )  
  \oplus  H^{k-\frac{1}{2}}(\Gamma) \,.
\end{equation*}
\end{remark}

The spaces $V_k$ and $V_k^-$ form an increasing scale of spaces. Namely:

\begin{remark}\label{rem.def.I} 
For suitable $F \subset \Gamma \cup \pa U_0$, we let $$\scal{u}{v}_{F} 
\ede \int_{F} u(x) \overline{v(x)} dS(x)\,,$$ where
$S$ is the induced volume form on $\Gamma \cup \pa U_0$.
Let $k \ge 0$. It follows immediately from definitions that $V_{k+1} \subset V_k$
Similarly, $V_{k+1}^- \subset V_k^-$, for $k \ge 1$. Let
$$I : V_1^- := L^2(U) \oplus  H^{\frac{1}{2}}(\partial_N) \oplus H^{\frac{1}{2}}(\Gamma)
\to V_{0}^{-} := V_0^* := H_{D}^{1}(U_0)^*$$ be given, for $v \in H_{D}^{1}(U_0)$, by
\begin{equation*}
  \langle I(f, g, h), v \rangle \ede \scal{f}{v} + 
  \scal{g}{v}_{\pa_N U_0} + 
  \scal{h}{v}_{\Gamma} \,.
\end{equation*}
Then $I : V_1^- \to V_0^-$ is a continuous inclusion. 
\end{remark}

Before we introduce our operators of interest $\maP_k^A$, let us introduce
the conormal derivative $D_\nu^A$.

\begin{definition} \label{def.derive.conormal}
Let $A \ede [a_{ij}] \in M_{m+1}(W^{1, \infty}(U_0))$  and
let $\nu$ be as in Assumptions~\ref{assum.decomposition}.
The \emph{conormal derivative associated to $A$} is the operator 
$D_{\nu}^{A}  : \check H^2(U_0) \to \check H^{1}(U_0),$
\begin{equation*} 
   D_{\nu}^{A} u \ede \sum_{i=1}^{m} \nu_i 
   \Big ( \sum_{j=0}^{m} a_{ij} \partial_j u \Big )  \,.
\end{equation*}
\end{definition}

Since $\pa_0 = id$, we have the more explicit formula
\begin{equation} \label{eq.def.derive.conormal}
   D_{\nu}^{A} u \ede \sum_{i=1}^{m} \nu_i 
   \Big ( \sum_{j=1}^{m} a_{ij} \partial_j u  + a_{i 0} u \Big )  \,.
\end{equation}

\begin{remark}\label{rem.saut.pan}
If $u \in \check{H}^2(U_0)$, then $D_\nu^{A} u \in \check{H }^1(U_0)$, 
so $D_{\nu}^{A}  : \check H^2(U_0) \to \check H^{1}(U_0)$ 
is well defined because we have assumed that $\nu \in \check W^{1, \infty}(U_0)^m$. Therefore the jump 
$[[D_\nu^{A} u]]_\Gamma \in H^{\frac12}(\Gamma)$ and the trace $D_\nu^{A} u\vert_{\partial \Omega}
 \in H^{\frac12}(\pa U_0)$ are defined.
\end{remark}

It is these last two terms of the above remark that are of greatest interest to us
since they appear in the following simple integration by parts lemma, and then
in the classical formulation of our transmission/boundary value problem,
see Remark~\ref{rem.classical}.

\begin{lemma} \label{lemma.int.partie} 
Let $B = B^{A}$ and $P = P^{A}$ be as in the Definition 
\ref{def.opsP} and $D_\nu = D_\nu^{A}$ be as in Definition~\ref{def.derive.conormal}. 
If $A \in M_{m+1}(\check W^{1, \infty}(U_0))$, then, for all $u \in \check{H}^2(U_0)$ 
and $v \in H^1(U_0) $, we have
\begin{equation*} 
    B(u, v) \seq \scal{Pu}{v}  + 
    \scal{D_\nu u}{v}_{\partial U_0}
    + \scal{[[D_\nu u]]_\Gamma}{v}_\Gamma \,.
\end{equation*}
\end {lemma}

\begin{proof} 
The proof is a standard integration by parts, as in~\cite{theseHassan, HassanSimonVictor1}.
\end{proof}

We are finally ready to introduce our operators of interest $\maP_k^A$.

\begin{definition} \label{def.opsP.k}
Let $B = B^A : H^1 (U_0)^2 \to \CC$ and $P = P^A : H^1(U_0) \to H^{-1}(U_0)$  
be as in Definition~\ref{def.opsP} and let $A \ede [a_{ij}] \in M_{m+1}(W^{k,\infty}(U_0))$. We then
define $\maP_k^A : V_k \to V_k^-$ as follows:
\begin{enumerate}
\item If $k =0$, then $\maP_0^A : V_0 \to V_0^-\ede V_0^*$ is given by duality by 
$$\langle \maP_0^*u, \phi \rangle_{V_0, V_0^*} \ede B^A(u, \phi)\,.$$

\item If $k \ge 1$, let $D_\nu^A$ be as in Definition~\ref{def.derive.conormal},
then 
\begin{equation*}
     \maP_k^A u \ede (P^A u, D_{\nu}^{A} u\vert_{\pa_N U_0}, 
    [[D_{\nu}^{A} u]]_\Gamma) \,.
\end{equation*}
\end{enumerate}
\end{definition}

The superscripts $A$ will usually be omitted. We see that, essentially, $\maP_k^A$ is (a restriction of)
$\maP_0^A$. The indices $k$ may thus also be omitted. The following two remarks discus the relation
between the various operators $P^A, \maP^A,$ and $\maP_k^A$, $k \in \ZZ_+$, introduced so far.

\begin{remark} \label{rem.relation0}
Of course, for $k \ge 1$, the operator $\maP_{k+1} = \maP_{k+1}^A$ is obtained from 
$\maP_k = \maP_k^A$ by restriction to $V_{k+1}$, using to the inclusions $V_{k+1} \subset V_{k}$ and 
$V_{k+1}^- \subset V_{k}^-$. In particular, $\maP_k^A u = \maP_{k+1}^A u$ for $u \in V_{k+1}$ and $k \ge 1$.
In fact, the same relation holds also for $k = 0$, using inclusion 
$I : V_1^- \to V_0^-$ defined in Remark~\ref{rem.def.I}. 
Indeed, Lemma~\ref{lemma.int.partie} gives, for all $u \in V_1$ and $v \in V_0$:
\begin{align*} 
    \langle  \maP_0 u, v \rangle_{V_0^*, V_0}  & \, =: \, B(u, v) \\
    &  \seq \scal{Pu}{v}  +  \scal{D_\nu^{A} u}{v}_{\partial U_0}
    + \scal{[[D_\nu^{A} u]]_\Gamma}{v}_\Gamma \\
   &  \seq \scal{Pu}{v}  +  \scal{D_\nu^{A} u}{v}_{\pa_N U_0}
    + \scal{[[D_\nu^{A} u]]_\Gamma}{v}_\Gamma \\
    & \, =: \, \langle I( \maP_1 u), v \rangle_{V_0^*, V_0} 
    \,.
\end{align*}
Thus $I( \maP_1 u) =  \maP_0 u$ for all $u \in V_1$. This relation is crucial, since
the invertibility is easiest to prove for the operator $\maP_0$, whereas is it needed
in applications for $\maP_k$, $k \ge1$, which involves, in addition to the invertibility
of $\maP_0$, a \emph{regularity} result.
\end{remark}

Let us now discuss the relation between the
operators $\maP_k^A$, $k \ge 0$, and the operators $\maP^A$ and $P^A$.

\begin{remark}\label{rem.relation1}
Let $j_0^* : H^1(U_0)^* \to H_D^1(U_0)^* =: V_0^-$  be the dual of the inclusions 
$V_0 := H_D^1(U_0) \subset H^1(U_0)$. Then $\maP_0^A = j_0^* \circ \maP^A\vert_{V_0}$.
Thus, the operator $\maP^A$ determines all the operators $\maP_k^A$, $k \ge 0$.
If there are no Dirichlet boundary conditions (i.e. $\pa_D U_0 = \emptyset$), then $\maP_0^A = \maP^A$,
but, in general, this is not true. Of course, when $\Gamma$ and 
$\pa U_0$ are empty, then $\maP^A = P^A$, as in~\cite{HassanSimonVictor1}, but this case is excluded in
this paper since we are assuming $U_0$ to be compact. In fact, in this paper, none of the operators
$\maP_k^A$ or $\maP^A$ coincides with $P^A$, the ``plain'' differential operator associated to our problem.
\end{remark}

The following remark makes the link between the objects we have introduced so far
and the classical transmission/mixed boundary value problems (Neumann boundary conditions 
on~$\partial_{N} U_0$ and Dirichlet boundary conditions on~$\partial_{D} U_0$).

\begin{remark} \label{rem.classical}
Let $F = (f, g, h) \in V_1^- = L^2(U_0) \oplus H^{1/2}(\pa_N U_0 
\cup \Gamma)$. For $u \in V_1 \subset \check{H}^2(U_0)$, the equation $\maP_1^A u = F$ 
is equivalent to the system~\eqref{eq.classical} of the Introduction.
Thus, if $ \maP_1 : V_1 \to V_1^-$ is invertible, then the solution to this system is therefore 
$u :=  \maP_1^{-1}F$, where $F:=(f,g,h) \in V_{1}^{-}$. If $\maP_0$ is invertible
and the coefficients $A$ are smooth enough, then it is know classically that $\maP_k$
is also invertible, at lest in the classical case of a trivial interface
$\Gamma = \emptyset$, see~\cite{BrezisBookFA, Evans, LionsMagenes1, TaylorPDEI}. See 
\cite{LiNistorQiao, NicaiseBook} and the references therein for the less classical case of a 
non-trivial interface $\Gamma$. Thus, if $F:=(f,g,h) \in V_{k}^{-} 
\subset V_1^{-} \subset V_0^{-}$, then $\maP_k^{-1}F = \maP_1^{-1}F = \maP_0^{-1}F$.
\end{remark}

The purpose of this paper is to obtain a \emph{polynomial estimate in $\|A\|$ and $\|\maP_0^{-1}\|$}
for the norm of $\maP_k^{-1}F$, where $F:=(f,g,h) \in V_{k}^{-}$. To that end, we want to use the approach 
in~\cite{Victor, theseHassan, HassanSimonVictor1}. Some modifications are needed, however, and the next 
sections begin to introduce some of these modifications.

\subsection{Nirenberg's trick}  
We shall need a version of Nirenberg's lemma (or ``trick''), as it is formalized in 
\cite{Victor, HassanSimonVictor1}. Let us first recall some classical definitions. 
We let $\maL(\maX; \maY)$ denote the set of continuous, linear
maps $\maX \to \maY$ between two normed spaces.

\begin{definition}\label{def.conv.so}
Let $\maX$ and $\maY$ be two normed spaces and 
$T_t \in \maL(\maX; \maY)$, $t \ge 0$. We say that $T_t$ \emph{converges strongly}
to $T$ in $\maL(\maX; \maY)$ for $t \searrow 0$ (i.e. $t \to 0$, $t > 0$), if
\begin{align*}
    \forall x \in \maX \,,  \quad \lim_{t \searrow 0} \|T_t x - Tx  \|_{\maY} = 0 \, .
\end{align*}
\end{definition}

We shall need also the following well-known concept~\cite{AmannBook}.

\begin{definition}\label{def.sg}
A family of operators $(S(t))_{t \geq 0} \in \maL(X) := \maL(\maX; \maX)$
is a \emph{strongly continuous semigroup} on $\maX$ if:
\begin{enumerate}[(i)]
\item $S(0)=id_{\maX}$ (the identity map).
\item For all $t, t'\geq0$, we have $S(t+t')=S(t)S(t')$ .
\item $S(t)$ converges strongly to $S(0) = id_\maX$ in $\maL(\maX)$ for $t \searrow 0$.
\end{enumerate}
\end{definition}

\begin{definition}\label{def.gen.inf}
Let $S(t) \in \maL(\maX)$ be a strongly continuous semigroup on a Banach space. 
Its \emph{infinitesimal generator} is the (possibly unbounded) operator
$$L_S \ede \lim_{t \rightarrow 0}\, t^{-1}(S(t)x-x)$$ with domain $\maD(L_S) \subset \maX$ 
the set of $x \in \maX$ for which this limit exists.
\end{definition}

We shall make essential use of the following lemma. (It is a form of ``Nirenberg's trick'', 
see also~\cite{Victor, elast} for further references.)

\begin{lemma}\label{astuceniren}
Let $S_\maX(t) \in \maL(\maX)$ and $S_\maY (t) \in \maL(\maY)$ 
be two strongly continuous semi-groups of operators on Banach spaces 
with infinitesimal generators $L_\maX$ and $L_\maY$, respectively. Let us assume that:
\begin{enumerate}[(i)]
\item $T \in \maL(\maX; \maY)$ is invertible, 

\item for all $t \ge 0$, there exists $T_t \in \maL(\maX, \maY)$ such that 
$T_t S_\maX(t) = S_\maY (t)T$, and 

\item $t^{-1} (T_t-T)$ converges strongly to $Q \in \maL(\maX; \maY)$ for $t \searrow 0$. 
\end{enumerate}
Then $T^{-1}(\maD(L_\maY)) \subset \maD(L_\maX)$, and, most importantly, 
for all $v \in \mathcal{D}(L_\maY)$, we have $$L_\maX (T^{-1} v) \seq 
T^{-1}L_\maY v - T^{-1} Q T^{-1} v \,.$$
\end{lemma}

See~\cite{Victor, HassanSimonVictor1} for a proof.
It also follows from the proof that the induced operator $ \tilde T:= T\vert_{\maD(L_\maX)}  : 
\maD(L_\maX)\rightarrow \maD(L_\maY)$ is well-defined, continuous, and bijective, but this
fact will not be used in this paper.

\section{Localization, Sobolev spaces, and diffeomorphism groups}  
\label{sec3}

In order to be able to use the approach in~\cite{Victor, theseHassan, HassanSimonVictor1}, 
as explained in the last section, we
want to localize our functions and operators to coordinate neighborhood
patches. This is the usual approach when dealing with regularity questions.
We are faced however, with the following problem. Although we 
can localize the definitions of Sobolev spaces and of the operator 
$\maP_0$, we cannot do the same for $\maP_0^{-1} $. So, we 
will need a localization construction for derivatives (or vector fields) that will replace the 
canonical derivatives $\pa_j$ coming from coordinates with other vector fields. 
This accounts for the greatest difference 
between this paper and~\cite{theseHassan, HassanSimonVictor1}. It is linked to 
the fact that we must use general vector fields to study our Sobolev spaces (instead 
of the standard derivatives $\partial_\ell$, as in~\cite{Mirela1}).

\subsection{Localization of vector fields and of norms}
Let $\overline{U}_0 = \bigcup_{j=1}^N \overline{U}_j$, as before, see Assumptions 
\ref{assum.decomposition}. Recall that all boundaries $\partial U_j$ are smooth, 
$j = 0, 1, \ldots, N$. (This implies that the 
boundary $\partial U_0$ of our domain and the interface $\Gamma := (\bigcup_{j=1}^N \partial U_j ) 
\smallsetminus \pa U_0$ are smooth and closed.)  We let $\RR_\pm^m := \{ \pm x_m \ge 0\}$
be the two half-spaces defined by the hyperplane $\{x_m = 0\}$.
The following objects and assumptions will remain in place throughout the rest of
this paper (some of which have already been used).

\begin{notation} \label{not.new}
We shall use the following notation for the following fixed from now on objects:
\begin{enumerate}
\item $\maV := (V_\alpha)_{\alpha \in I}$ is a finite, open covering of $\overline{U}_0$
(so $V_\alpha \subset \RR^m$ is open and $\overline{U}_0 \subset \bigcup_\alpha V_\alpha$).

\item $\phi_\alpha : W_\alpha \stackrel{\sim}\longrightarrow W_\alpha' \subset \RR^m$
are some fixed from now on diffeomorphisms (coordinate charts) of open subsets with 
$\overline{V}_\alpha \subset W_\alpha$.

\item $\Gamma_k$ are the connected components of $\Gamma$ (so $\Gamma = \cup \Gamma_k$, a disjoint union).
\end{enumerate}
\end{notation}

In addition to this specific notation, we shall let $\CIc(W)$ denote the set of smooth, compactly supported
functions on an open set $W$, as usual. The objects introduced in Notation~\ref{not.new} satisfy the following properties.

\begin{assumptions} 
\label{assumpt.new}
In addition to Assumptions~\ref{assum.decomposition}, we assume that the objects introduced in Notation~\ref{not.new}
satisfy the following properties:
\begin{enumerate}
  \item If $W_\alpha \cap \partial U_0 \neq \emptyset$, 
  then  $W_\alpha \cap \Gamma = \emptyset$ and 
\begin{equation*}
    \phi_\alpha( W_\alpha \cap U_0 ) \subset \RR_+^m \quad \text{ and } \quad
    \phi_\alpha(W_\alpha \cap \partial U_0) \subset \RR^{m-1} \seq \partial \RR_+^m \,.
\end{equation*}

\item If $W_\alpha \cap \Gamma \neq \emptyset$, then
$W_\alpha \cap \partial U_0 = \emptyset$, there is only one $k$ such that 
$W_\alpha \cap \Gamma_k  \neq \emptyset$, and
\begin{equation*}
    \phi_\alpha(W_\alpha \cap \Gamma_k) \subset \RR^{m-1} \seq \{ x_m = 0 \} \,.
\end{equation*}
Moreover, if $\Gamma_k \subset
\pa U_j$, then, for some $\epsilon \in \{+, -\}$, we have
\begin{equation*}
    \phi_\alpha(W_\alpha \cap U_j) \subset \RR_\epsilon^{m-1} \ede \{ \epsilon x_m \ge 0 \} \,.
\end{equation*}

\item $\tilde \psi_\alpha \in \CIc(W_\alpha)$ is such that $\tilde \psi_\alpha = 1$ on $V_\alpha$.

\item We assume that both $\phi_\alpha$ and $\phi_\alpha^{-1}$ have components in $W^{\infty, \infty}$.
\end{enumerate}
\end{assumptions}

Thus each of the sets $W_\alpha$ interesects at most one of $\pa U_0$ or $\Gamma_k$. It also
follows that $\overline{V}_\alpha \subset W_\alpha$, so $\maW := (W_\alpha)_{\alpha \in I}$
is also an open covering of $\overline{U}_0$. Such finite coverings $\maV := (V_\alpha)_{\alpha \in I}$ 
and $\maW := (W_\alpha)_{\alpha \in I}$ and the functions $\tilde \psi_\alpha$ exist by standard 
topology results since $U_0$ is bounded and all boundaries $\pa U_j$, $j=0, \ldots, N$, were
assumed to be smooth, see Assumptions~\ref{assum.decomposition}.

\subsubsection{Localizing vector fields} 
The role of the new coordinates $\phi_\alpha$, $\alpha \in I$, is to recover the setting considered 
in~\cite{HassanSimonVictor1, theseHassan}. That setting has already been dealt with in a uniform,
comprehensive way. Most importantly, it allows us to replace the partial derivatives $\pa_j$ 
\cite{HassanSimonVictor1, theseHassan} with
some other vector fields that we introduce next.

\begin{remark}\label{rem.def.X.vect}
See Notation~\ref{not.new} and Assumptions~\ref{assumpt.new} for the notation and let, for each $\alpha \in I$,
\begin{equation*}
\begin{gathered}
   Y_j^\alpha \ede \phi_{\alpha*}^{-1}(\pa_j)  \quad \mbox{and}    \\   
   X_j^\alpha \ede \tilde \psi_\alpha Y_j^\alpha \seq \tilde \psi_\alpha \phi_{\alpha*}^{-1}(\pa_j)\,,
\end{gathered}
\end{equation*}
where $\phi_{\alpha*}^{-1}(\pa_j)$ is the push-forward of the standard vector field $\pa_j := \frac{\pa}{\pa x_j}$, 
$j = 1, \ldots, m$, via the differentiable map $\phi_\alpha^{-1}$. (More explicitly, $X_j^\alpha \seq \tilde \psi_\alpha 
\sum_{i=1}^m \pa_j( \phi^{-1}_{\alpha i}) \pa_i,$ where $\phi^{-1}_{\alpha j}$ are the components of $\phi^{-1}$.)
We also obtain that $X_j^\alpha$ is tangent to the boundary of~$U_k$ for all $j < m$ and all $k$. 
Also, all vector fields $X_j^\alpha$, $j = 1, \ldots, m$, $\alpha \in I$, 
are defined \emph{everywhere} on $\overline{U}_0$ (unlike the vector fields $Y_j^\alpha$,
which will play an auxiliary role). For the same reason for which we
have introduced $\pa_0 := id$, we let $X_0^\alpha = id$, since this will simplify some
formulas.
\end{remark}

\subsubsection{Localizing Sobolev spaces}
We can localize well also the definitions of our Sobolev spaces In the standard way.

\begin{remark}\label{rem.def.norm.sob.var} 
We continue to use the notation of Notation~\ref{not.new} and Assumptions~\ref{assumpt.new}. 
The subdomains $(U_j)_{j=1}^{N}$ of $U_0$ and the given coordinate charts 
(diffeomorphisms) $\phi_\alpha : W_\alpha \to W_\alpha'$ induce decompositions on 
$W_\alpha^\prime = \phi_\alpha(W_\alpha)$, which hence define broken Sobolev spaces on 
these domains. Let $V_{\alpha}^{\prime}  := \phi_\alpha(V_\alpha)$. Then the finite, open cover 
$\maV = (V_\alpha)_{\alpha \in I}$, the coordinate charts  $\phi_\alpha$, and the functions
$\tilde \psi_\alpha$ define 
equivalent norms on the broken Sobolev spaces $\check{H}^{s}(U_0) \ede \bigoplus_{j=1}^N H^{s}(U_j)$ 
and $\check{W}^{s, \infty}(U_0) \ede \bigoplus_{j=1}^N W^{s, \infty}(U_j)$:
\begin{equation*} 
\begin{gathered}
  \nH{s}{u}
  \ede \sum_{\alpha \in I } \| u \circ \phi_\alpha^{-1} 
  \|_{\check H^s(V_\alpha')}\\
%
  \nW{s}{u}
  \ede \max_{\alpha \in I} \| u \circ \phi_\alpha^{-1} 
  \|_{\check W^{s, \infty}(V_\alpha')}
  \,, \quad s \ge 0\,.
\end{gathered}  
\end{equation*}
The fact that these norms are equivalent to the original ones is a very standard 
calculation using the fact that we have a finite covering and that
each diffeomorphism $\phi_\alpha\vert_{V_\alpha} : V_\alpha \to \phi_\alpha(V_\alpha) \subset
W_\alpha'$ extends (obviously) to the larger open subset $W_\alpha$ containing $\overline{V}_\alpha$. 
Very similar proofs can be found, for instance, in~\cite{babuska} in a similar setting 
and in~\cite{AGN1, Grosse, TriebelBG} for infinite coverings in the framework of bounded geometry.
\end{remark}


In order to define similar equivalent norms for the coefficient matrices,
it will be convenient to use the following matrix norm:

\begin{definition}\label{def.normop.var}
Let $\maO \subset U_0$. For a matrix $A = [a_{ij}] \in M_{m+1}(\check W^{k,\infty}(\maO))$, we let
\begin{align*}
  \|A\|_{\check W^{k,\infty}(\maO)}  := \max_{0 \leq i \leq m} \Big \{ \sum_{j=0}^m 
  \| a_{ij} \|_{\check W^{k,\infty}(\maO)} ,  \sum_{j=0}^m \| a_{ji} \|_{\check W^{k,\infty}(\maO)} 
  \Big \} \ .
\end{align*}
\end{definition}

We next define analogous norms to those introduced in Remark~\ref{rem.def.norm.sob.var}.
Let $\tilde \psi_\alpha$ be as in Assumptions~\ref{assumpt.new} (and in the definition of the vector fields $X_j^\alpha$,
Remark~\ref{rem.def.X.vect}).

\begin{remark}\label{rem.def.norm.A.var}
Let $A^\alpha \in M_{m+1}(W^{k,\infty}(W_\alpha'))$ be the coefficients corresponding to the form 
induced by $B^A$ on $\CIc(W_\alpha')$. (The precise definition of $A^\alpha$, while standard, is technical and hence 
is recalled after the following definition.) We let 
\begin{equation*}
   \nnW{k}{A} \ede \max_{\alpha \in I} \| A^\alpha \|_{\check{W}^{k} (W_\alpha') }.
\end{equation*}
As in Remark  the norms $\nnW{k}{\cdot}$ and $\| \cdot \|_{\check W^{k,\infty}}$ are equivalent on 
$\maL(V_k; V_k^-)$.
\end{remark}

\subsubsection{Localizing the cofficients $A$}\label{sssec.aalpha}

Let $P^A := \sum_{i,j = 0}^m \pa_i^* a_{ij} \pa_j$ as before
(see also Remark~\ref{rem.d0} for notation and further explanations). 
For each $\alpha \in I$, the diffeomorphism $\phi_\alpha : W_\alpha \to  W_\alpha'$ induces an
isomorphism $\phi_\alpha^* : \CIc( W_\alpha') \to  \CIc( W_\alpha)$ by $\phi_\alpha^*(f) := f \circ \phi_\alpha$.
A vector field $X$ on $W_\alpha$ can be defined as a derivation $X : \CIc(W_\alpha) \to \CIc(W_\alpha)$.
As such, it can uniquely be written as $X = \sum_{j=1}^m X_j \pa_j$, with $X_j \in \CI(W_\alpha)$.
It induces a vector field 
\begin{equation}\label{eq.expansion}
     \phi_{\alpha *}(X) := \phi_\alpha^{*-1} \circ X \circ \phi_\alpha^* =: Y = \sum_{j=1}^m Y_j \pa_j
\end{equation}
on $W_\alpha'$. We can apply this construction to any differential operator
$Q : \CIc( W_\alpha') \to \CIc( W_\alpha')$ to obtain a differential operator
\begin{equation}\label{eq.def.mor.phi}
    \phi_{\alpha *}(Q) := \phi_\alpha^{*-1} \circ Q \circ \phi_\alpha^* : \CIc( W_\alpha') \to \CIc( W_\alpha')
\end{equation}
In particular, we can do that for our differential operator $P^A$
(or, more precisely, to its restriction to $\CIc( W_\alpha)$) to obtain a differential operator 
$\phi_{\alpha *}(P^A) : \CIc( W_\alpha') \to \CIc( W_\alpha')$. It is,
of course of the form $P^{A^\alpha} \in M_{m+1}(W_\alpha')$, but $A^\alpha$ is not uniquely determined
(see Remark~\ref{rem.d0}). However, by using that $\phi_{\alpha *}$ is an algebra morphism, the unique
writing of Equation~\eqref{eq.expansion} applied to each of the vector fields $\phi_{\alpha *}(\pa_j)$,
$j = 1, \ldots, m$, and that $\phi_{\alpha *}(\pa_0) = \pa_0 = id$, we obtain a \emph{canonical}
$A^\alpha = [a_{ij}^\alpha]$ :
\begin{align*}
     \phi_{\alpha *}(P^A)\ & \seq \sum_{i,j = 0}^m \phi_{\alpha *}(\pa_i^*) \, (a_{ij} \circ \phi_\alpha^{-1})\, 
     \phi_{\alpha *}(\pa_j) \\
     & \seq \sum_{i,j = 0}^m \pa_i^*\, a_{ij}^\alpha\, \pa_j \ = : P^{A^\alpha} \ \,.
\end{align*}
In particular, $A^\alpha$ depends linearly on $A$. Indeed, let $\phi_{\alpha}(\pa_i) = \sum_{j=1}^m Z_{ij} \pa_j$
and $\phi_{\alpha}(\pa_i^*) = \sum_{j=1}^m Z_{ij}^* \pa_j$. Let $Z_{00} = Z_{00}^*$ and $Z_{0i} = Z_{0i}^* =
Z_{j0} = Z_{j0}^*$ if $i, j \ge 1$. Then
\begin{equation}\label{eq.def.Aalpha}
     a_{kl}^{\alpha} \seq \sum_{ij} Z_{ik}^* a_{ij} \circ \phi_\alpha^{-1} Z_{jl} \,.  
\end{equation}
By applying the inverse morphis $\phi_{\alpha *}$
to the above equation, we obtain

\begin{lemma} \label{lemma.local.alpha}
Recall that $Y_j^\alpha := \phi_{\alpha *}^{-1}(\pa_j)$, $j = 1, \ldots, m$. 
Let $Y_j^{\alpha *} := -Y_j^\alpha$, if $j = 1, \ldots, m$.
Let also $Y_0^\alpha = Y_0^{\alpha*} = id$. Then, the canonical
matrix $A^\alpha = [a_{ij}^\alpha]$ of Equation~\eqref{eq.def.Aalpha} satisfies
$A^\alpha \in M_{m+1}(\check W^{k,\infty}(W_\alpha'))$ and
\begin{equation*}
     P^A \seq \sum_{i, j = 0}^m Y_i^{\alpha *} \, (a_{ij}^\alpha \circ \phi_\alpha)\, Y_j^\alpha
\end{equation*} 
as differential operators on $W_\alpha$.
\end{lemma}

\subsection{Sobolev spaces and vector fields}
\label{sec.sobolev}

It will be convenient to use the following shorthand notation for
the following two norms that often appear in our results.

\begin{notation}\label{not.def.normop2} We let
\begin{equation*}
    \| T_1 \|_k \ede \| T_1 \|_{\maL(V_{k}; V_{k}^-)} \qquad \text{and}
    \qquad ||| T_2|||_k \ede \| T_2 \|_{\maL(V_{k}^-; V_{k})} \,.
\end{equation*}
\end{notation}

We shall write $a \lesssim b$ if there is a $C$ that depends only on the order of the
Sobolev spaces, on $\maV$, and $m$ such that $a \le C b$. In particular, in these
inequalities, $C$ is
independent of the variables belonging to the various functions spaces involved.

\begin {lemma} \label{lemma.cont} Let $k \ge 0$.
Let $a \in \check W^{k,\infty}(U_0)$
and $f \in \check H^{k}(U_0)$. Then, $a f \in \check H^{k}(U_0)$
and $\|af\|_{\check H^k} \lesssim \|a \|_{\check W^{k,\infty}} \|f\|_{\check H^k}$.
Consequently, 
\begin{align*}
   \| \maP_k \|_k \lesssim \|A\|_{\check W^{k,\infty}} \quad \mbox{and} \quad \| \maP_k \|_k \lesssim \nnW{k}{A}\,.
\end{align*}
\end {lemma}

\begin{proof}
The first part is well known. The second part is a consequence of the first part. The 
third part is a consequence of the first part and of 
the definitions of the norms in local coordinates. Indeed
\begin{equation*}
   \|\maP_k^A u \|_{V_k^-} \leq \|\maP_k^A\|_k \|u\|_{V_k} \lesssim \|A\|_{W^{k,\infty}} \|u\|_{V_k}
   \lesssim \nnW{k}{A} \|u\|_{V_k} \,,
\end{equation*}
where the last inequality is due to the equivalence of the norms $\| \cdot \|_{\check W^{k,\infty}}$
and $\nnW{k}{\cdot}$ of Remark
\ref{rem.def.norm.A.var}.
\end{proof}

The biggest difference between this paper and~\cite{HassanSimonVictor1}
is that we must use vector fields to study our Sobolev spaces
(instead of standard derivatives~$\partial_\ell$). Moreover, well
that we can localize the definitions of the Sobolev spaces and of the operators
$\maP_k$, we cannot localize the definition of~$\maP_0^{-1}$.

\begin{definition}\label{def.amm} 
Let $A \in M_{m+1}(\check W^{k,\infty}(U_0))$ and
$A^\alpha = [a_{ij}^\alpha]  \in M_{m+1}(\check W^{k,\infty}(W_\alpha'))$ be as in Remark
\ref{rem.def.norm.A.var}  
\begin{equation*}
   \nnW{k}{a_{mm}^{-1}} \ede \max_\alpha 
   \| \big (a_{mm}^\alpha \big)^{-1} \|_{\check{W}^{k} (W_\alpha') } \,.
\end{equation*}
Note also that one may have $\nnW{k}{a_{mm}^{-1}} := \infty$ (in which case $a_{mm}^\alpha$ 
is not invertible everywhere on 
all $W_\alpha$).
\end{definition}

We will need the following results which give new descriptions
broken Sobolev spaces. We also set $X_0^\alpha = id$, to simplify the notation.

\begin{lemma}\label{lemma.product.ge2}
Let $A \in M_{m+1}(\check W^{k,\infty}(U_0))$ and
$A^\alpha = [a_{ij}^\alpha] \in M_{m+1}(\check W^{k,\infty}(W_\alpha'))$ be as in Remark
\ref{rem.def.norm.A.var} and
$\maP_k := \maP_{k}^{A} : V_{k} \to V_{k}^-$ the operators associated to $A$ (see Section~\ref{ssec.Pk}),
which is assumed to be invertible. Hence
\begin{align*}
    1 \lesssim & |||  \maP_k^{-1} |||_{k}  \,  \nnW{k}{A} 
    \quad \text{ et }\ \quad 1 \le  \nnW{k}{a_{mm}^{-1}} 
    \nnW{k}{A} \lesssim \nnW{k}{a_{mm}^{-1}} 
    \| A\|_{\check W^{k,\infty}} \, .
\end{align*}
\end{lemma}

\begin{proof}  
The submultiplicativity of the norms, Lemma~\ref{lemma.cont},
and the equivalence of the norms of $A$ in 
Remark~\eqref{rem.def.norm.A.var} (in this order) give that
\begin{align*}
    1 \le ||| P^{-1} |||_{k}  \| P\|_{k} \lesssim |||  \maP_0^{-1} |||_{k}  \| A\|_{\check{W}^{k}}   
    \lesssim |||  \maP_0^{-1} |||_{k} \nnW{k}{A} \,.  
\end{align*}
The second part is similar.
\end{proof}

The following lemma is crucial computational step for our estimates, 
it generalizes Lemma 5.1 of~\cite{HassanSimonVictor1}. 
In the following, we include the domain $U_0$ in the norms on $U_0$, unlike
in most of the other formulas.

\begin {lemma} \label{lemma.sobolev1et2}
Recall the vector fiels $X_j^\alpha$, $j = 0, \ldots, m$, $\alpha \in I$ of
\ref{rem.def.X.vect} and that $X_0^\alpha = id$.  We use also the notation of 
Lemma~\ref{lemma.product.ge2}.
\begin{enumerate}
\item An equivalent norm on $\check H^{k+2}(U_0)$ is given by $$\|u\|' \ede \sum_{\alpha \in I}
\sum_{j = 0}^m \|X_j^\alpha u\|_{\check H^{k+1}(U_0)}\,.$$

\item Similarly, an equivalent norm on $\check H^{k+2}(U_0)$
is given by $$\|u\|'' \ede \sum_{\alpha \in I} \sum_{i, j=0}^m 
\|X_i^\alpha X_j^\alpha u\|_{\check H^{k}(U_0)} \,.$$

\item Let $A$ and $P = P^A := \sum_{i, j = 0}^m \pa_i^* a_{ij} \pa_j$ be as in 
Definition~\ref{def.opsP} and $\nnW{k}{a_{mm}^{-1}}$ 
be as in Definition~\ref{def.amm}. Then
\begin{equation*}
   \|u\|_{\check H^{k+2}(U_0)} \lesssim  
   \nnW{k}{a_{mm}^{-1}} \Big ( \|P u\|_{\check H^{k}(U_0)}
   + \|A\|_{\check W^{k+1,\infty}(U_0)} \sum_{j=0}^{m-1} \|X_j^\alpha u\|_{\check H^{k+1}(U_0)} \Big)   
\,.
\end{equation*}
\end{enumerate}
\end {lemma}

The point of (3) is that we are not using the vector fields $X_m^\alpha$.

\begin{proof}
The property~(1) is true because it is true (well known) in the Euclidean case and because we
have the norm equivalences of Remark~\ref{rem.def.norm.sob.var}. We include the details for
the convenience of the reader. Let $V_\alpha' := \phi_\alpha(V_\alpha)$,
where $\phi_\alpha$ are as in Notation~\ref{not.new} and Assumptions~\ref{assumpt.new}. 
We then have $\|X_j^\alpha u\|_{\check H^{k+1}} 
\lesssim \|u\|_{\check H^{k+2}}$, for all $\alpha \in I$ and
$j = 0, \ldots, m$, and hence $\|u\|' \lesssim \|u\|_{\check H^{k+2}}$. To prove
the other inequality, we proceed as follows:

\begin{align*}
    \| u \|_{\check H^{k+2}(U_0)} \, \lesssim \, \nH{k+2}{u} \, 
    & \ede \sum_{\alpha \in I } \| u \circ \phi_\alpha^{-1} \|_{\check H^{k+2}(V_\alpha')} \\
    &  \lesssim \ \sum_{\alpha \in I } \sum_{j=0}^m 
    \| \pa_j (u \circ \phi_\alpha^{-1}) \|_{\check H^{k+1}(V_\alpha')} \\
    &  \lesssim \ \sum_{\alpha \in I } \sum_{j=0}^m 
    \| (X_j^\alpha u) \circ \phi_\alpha^{-1}) \|_{\check H^{k+1}(V_\alpha')} \\
    &  \lesssim \ \sum_{\alpha \in I } \sum_{j=0}^m \nH{k+1}{X_j^\alpha u}  \lesssim \ \sum_{\alpha \in I } 
    \sum_{j=0}^m \| X_j^\alpha u \|_{\check H^{k+1}} \, =:\, \|u\|' \,,
\end{align*}
where in the second inequality we have used the result in the Euclidean case (in which case
it is simple, classical, and  well known, as we have mentioned already).

The proof of~(2) is completely similar, but ths time comparing with $$\sum_{\alpha \in I } \sum_{i, j=0}^m 
    \| \pa_i \pa_j (u \circ \phi_\alpha^{-1}) \|_{\check H^{k+1}(V_\alpha')}\,.$$

Let us now prove~(3). We may assume $A \in \check W^{k+1,\infty}(U_0)$ (otherwise the
result is trivial). For $u \in \CIc(V_\alpha)$, Lemma~\ref{lemma.local.alpha} gives
\begin{equation}
    Pu \seq P^Au \seq \sum_{i, j = 1}^{m} a_{ij}^\alpha X_i^\alpha X_j^\alpha + R^\alpha
\end{equation}
where $R^\alpha := \sum_{j=0}^m r_j^\alpha \pa_j u$ is a differential operator of order one,
with coefficients $r_{j}^\alpha \in \check W^{k,\infty}(U_0)$
and $\|r_{j}^\alpha \|_{\check W^{k,\infty}} \lesssim 
\|A\|_{\check W^{k+1,\infty}}$. Hence
\begin{equation}\label{Ru.var}
     \| R^\alpha u \|_{\check H^{k}} \lesssim  \|A\|_{\check W^{k+1,\infty}} 
     \|u \|_{\check H^{k+1}} \, .
\end{equation}
by Lemma~\ref{lemma.cont}. Moreover $a_{ij}^\alpha \in \check W^{k+1,\infty}(U_0)$
satisfy $\| a_{ij}^\alpha \|_{\check W^{k+1,\infty}} \lesssim \| A \|_{\check W^{k+1,\infty}}$. 

Let $Y_j^\alpha$ be as defined in Remark~\ref{rem.def.X.vect} and used in Lemma~\ref{lemma.local.alpha}.
Let us write $\widetilde{\sum}_{i+j < 2m}$ for the sum over all pairs $(i, j)$ such that
$i, j \ge 1$ and $i + j < 2m$. For $u \in \CIc(W_\alpha)$, Lemma~\ref{lemma.local.alpha} gives
\begin{equation}\label{eq.Ym}
   Y_m^\alpha Y_m^\alpha u \seq (a_{mm}^\alpha)^{-1} \big ( R^\alpha u 
   - \widetilde{\sum}_{i+j<2m} a_{ij}^\alpha Y_i^\alpha Y_j^\alpha u  - Pu \big)   \,.
\end{equation}
Since $X_j^\alpha := \tilde \psi_\alpha Y_j^\alpha$ with $\tilde \psi_\alpha \in 
\CIc(W_\alpha)$, the inequality
$1 \lesssim  \nnW{k}{a_{mm}^{-1}}   \|A\|_{\check W^{k+1,\infty}}$
of Lemma~\ref{lemma.product.ge2}, the inequality~\eqref{Ru.var}, Lemma~\ref{lemma.cont}, 
Equation~\eqref{eq.Ym}, and the relation $Y_i^\alpha Y_j^\alpha = Y_j^\alpha Y_i^\alpha$
imply that
\begin{align*}
   \| X_m^\alpha X_m^\alpha u \|_{\check H^{k}} \ & \lesssim \| 
   \tilde \psi_\alpha Y_m^\alpha Y_m^\alpha u \|_{\check H^{k}(W_\alpha)} + 
   \|u \|_{\check H^{k+1}(W_\alpha)}\\
   & \lesssim \ \nnW{k}{a_{mm}^{-1}}
  \| A \|_{\check W^{k+1,\infty}} \big (  \|u \|_{\check H^{k+1}}
   +  {\widetilde{\sum}_{i+j<2m}} \, \|\tilde \psi_\alpha Y_i^\alpha Y_j^\alpha u \|_{\check H^{k}} \big )
   + \| a_{mm}^{-1} Pu \|_{\check H^{k}} \\
   & \lesssim \ \nnW{k}{a_{mm}^{-1}}
  \| A \|_{\check W^{k+1,\infty}}  \big (  \|u \|_{\check H^{k+1}}
   +  \sum_{i=1}^{m-1} \|X_i^\alpha (u) \|_{\check H^{k+1}} \big ) 
   + \| a_{mm}^{-1} Pu \|_{\check H^{k}}\\
   & = \ \nnW{k}{a_{mm}^{-1}}
  \| A \|_{\check W^{k+1,\infty}}  \sum_{i=0}^{m-1} \|X_i^\alpha (u) \|_{\check H^{k+1}} 
   + \| a_{mm}^{-1} Pu \|_{\check H^{k}} \,. 
\end{align*}
The last inequality and part~(2) then give that
\begin{align*}
    \|u\|_{\check H^{k+2}} \ & \lesssim \ \sum_{\alpha \in I}\sum_{i,j=0}^m 
    \|X_i^\alpha X_j^\alpha u\|_{\check H^{k}}  \\
    & \lesssim \ \sum_{\alpha \in I} \big( \|X_m^\alpha X_m^\alpha u\|_{\check H^{k}} + 
    \widetilde{\sum}_{i + j \le 2m-1}
    \|X_i^\alpha X_j^\alpha u\|_{\check H^{k}} + \|u\|_{\check H^{k+1}} \big)\\
    & \lesssim \ \sum_{\alpha \in I} \big( \|X_m^\alpha X_m^\alpha u\|_{\check H^{k}} + \sum_{i = 1}^{m-1}
    \|X_i^\alpha u\|_{\check H^{k+1}} + \|u\|_{\check H^{k+1}} \big)\\
   & \lesssim \sum_{\alpha \in I} \Big(  \ \nnW{k}{a_{mm}^{-1}}
   \| A \|_{\check W^{k+1,\infty}}  \sum_{i=0}^{m-1} \|X_i^\alpha (u) \|_{\check H^{k+1}} 
   + \| a_{mm}^{-1} Pu \|_{\check H^{k}} \Big) \,,    
\end{align*}
where, for the last equation, we used again the relation
$1 \lesssim  \nnW{k}{a_{mm}^{-1}}   \|A\|_{\check W^{k+1,\infty}}$
of Lemma~\ref{lemma.product.ge2}). A final application of Lemma~\ref{lemma.cont} gives
$\| a_{mm}^{-1} Pu \|_{\check H^{k}} \le \nnW{k}{a_{mm}^{-1}} \| Pu \|_{\check H^{k}}$, and
hence (3).
\end{proof}

\subsection{Groups of diffeomorphisms}
We recall that a tangent vector $X \in T_y \RR^m$ is a derivation $X : \CI(\RR^m) \to \CC$
such as $X(fg) = X(f) g(y) + f(y) X(g)$ and $X(\overline f) = \overline{X(f)}$
(since we are working with complex valued functions but consider only vector fields with
real coefficients). A tangent vector field to $\RR^m$
is a family of tangent vectors $X(y) \in T_y \RR^m$, for all $y \in \RR^m$,
which varies in a smooth way. Since $U_0 \subset \RR^m$, we have that $X = \sum_{j = 1}^m
a_{j} \partial_j$, with the functions $a_{j}$ smooth, real valued.

\begin{remark}\label{rem.champ.X}
Let $X$ be a compactly supported smooth vector field on $\RR^m$ which is
\emph{tangent} to~$\partial U_0$. Since $X$ is compactly supported and tangent
to the boundary of $U_0$,
there is a one-parameter subgroup of
diffeomorphisms $\phi_t: \overline{U}_0 \to \overline{U}_0$ that
integrates $X$, in the sense that the derivative of $\phi_t(y)$ at $t = 0$ is
$X(y)$:
\begin{equation*}
 \lim_{t \to 0} t^{-1}(f(\phi_t(y)) - f(y)) \seq X(y) f \,.
\end{equation*}
We also have the group property: $\phi_t \circ \phi_s = \phi_{t+s}$,
for all $t, s \in \RR$. If moreover $X$ is \emph{tangent} at all boundaries
$\partial U_k$, $k=0, 1, \ldots, N$, then $\phi_t(U_k) = U_k$ for all
$k = 0, 1, \ldots, N$.  
\end{remark}

\begin{proposition}\label{prop.S}  
Let $X$ be a compactly supported, smooth vector field on $\RR^m$ that is
tangent to all boundaries $\partial U_j$, $j=0, 1, \ldots, N$. We let
$\tau_t$ denote the group of isomorphisms generated by $X$, that is, 
$\tau_t (f) \ede f \circ \phi_{t}$, where $\phi_t$ is as in Remark~\ref{rem.champ.X}. 
Let $k \ge 0$ and $\maP_k^A$ and the $V_k$ and $V_k^-$ be as in Section
\ref{ssec.Pk}. Then

\begin{enumerate}
\item $\tau$ defines a strongly continuous group $S(t)$ of operators on $V_k$ such that
$V_{k+1} \subset \maD(L_S)$ and $L_S = X$ on $V_{k+1}$.

\item Similarly, $\tau$ defines a strongly continuous group $S_-(t)$ of operators
on $V_k^-$ such that $V_{k+1}^- \subset \maD(L_{S_-})$ and $L_{S_-} = X$ on $V_{k+1}^-$.

\item The maps $X : V_{k+1} \to V_{k}$ and 
$X : V_{k+1}^- \to V_{k}^-$ induced by $X$ are well-defined and continuous.

\item \label{item.prop.S.h1} $\tau$ also defines a group of
isomorphisms (still denoted $\tau$) of $\check{W}^{k, \infty}(U_0)$
by the relation $S(t) \maP_k^A = \maP_k^{\tau_t(A)} S_-(t)$ for $t \in \RR$ and
$A \in M_{m+1}(\check{W}^{k, \infty}(U_0))$. 

\item \label{item.prop.S.h2}
If $A \in M_{m+1}(\check{W}^{k+1, \infty}(U_0))$,
then there exists$X(A) \in M_{m+1}(\check{W}^{k, \infty}(U_0))$ (loss of one derivative!)
such that
$$t^{-1}( \maP_k^{\tau_t(A)} - \maP_k^A) \to \maP_k^{X(A)}
\, =: \, X( \maP_k^A) $$ strongly in~$\maL(V_{k}; V_{k}^-)$.
\end{enumerate}
\end{proposition}

The precise formula for $X(A)$ is $$X(A) \seq \sum_{i, j=0}^m \big ([X, \pa_i^*] a_{ij} \pa_j
+ \pa_i^* X(a_{ij}) \pa_j + \pa_i^* a_{ij} [X, \pa_j] \big) \,.$$ 
We also note that our semigroup is \emph{not} strongly continuous on $\check{W}^{k, \infty}(U_0)$.
See Section~\ref{ssec.Pk} for the notation not explained in the statement.

\begin{proof}
The assumption that $\pa_D U_0$, $\pa_N U_0$, and $\Gamma$ are all closed and 
disjoint and the fact that the problem is local show that we can assume that 
that $X = \psi \pa_\ell$, with $\ell < m$ and $\psi \in \CIc(\RR^m)$ and that
\begin{itemize}
  \item either  $N = 1$ (so no interface), $U_0 = \RR^m_+$, and $X = \psi \pa_\ell$ 
  \item or that $N =2$ (so two subdomains), $U_0 = \RR^m$, $U_1 = \RR_+^m$, and $U_2 = \RR_-^m$.
\end{itemize}
Except for the factor $\psi$, this is the case treated in~\cite{HassanSimonVictor1}. 
We thus repeat the proof there with the small modifications warranted by the introduction of
the factor $\psi$.
As in~\cite{HassanSimonVictor1}, it suffices to
show that we have a continuous inclusion $V_1^- \subset
D(L_S)$ and the points~\eqref{item.prop.S.h1} and~\eqref{item.prop.S.h2}, because the rest is an immediate consequence
of the definitions. 

To show the continuous inclusion $V_1^- \subset D(L_S)$,
we will use the inclusion $I: V_1^- \to V_0^-$ of Remark~\ref{rem.def.I}.
For all $v \in H^1(U_0)$ and
$(f, g, h) \in V_1^- := L^2(U_0) \oplus H^{1/2}(\pa_N U_0)
\oplus H^{1/2}(\Gamma)$, we have 
\begin{multline*}
   \langle t^{-1}(S(t) - 1) I(f, g, h), v \rangle 
   \seq \langle  I(f, g, h), t^{-1}(S(t)^* - 1) v \rangle\\
   \ede \scal{f}{t^{-1}(S^*(t) - 1) v} + 
   \scal{g \oplus h}{t^{-1}(S^*(t) - 1) v }_{\pa_N U_0 \cup \Gamma}\\
   \to  - \scal{f}{   \pa_\ell (\psi v)} - 
   \langle g,  \pa_\ell ( \psi v \vert_{\pa_N U_0}) \rangle_{\tilde \psi  \pa_N U_0} - 
   \langle h,    \pa_\ell (\psi v\vert_{\Gamma}) \rangle_{\Gamma} \\
   =:\,  \langle  \psi  \pa_\ell I(f, g, h), v \rangle    
   \seq  \langle\tilde \psi   \pa_\ell f, v \rangle_{U_0} + 
   \langle \tilde \psi  \pa_\ell g, v\rangle_{ \tilde \psi  \pa_N U_0} + 
   \langle\tilde \psi  \pa_\ell h, v\rangle_{\Gamma}\,,
\end{multline*}
where $\langle \cdot , \cdot \rangle_{\Sigma}$ is the duality between distributions and test-functions
on the set $\Sigma$ (which is an extension of the scalar product $\scal{\cdot}{\cdot}_\Sigma$,
according to our conventions). So 
$I(f, g, h)$ is in the domain of the infinitesimal generator $L_{S_-}$ of $S_-$,
$$L_{S_-} I(f, g, h) \seq  \tilde \psi \pa_\ell I(f, g, h) \seq X I(f, g, h)\,,$$ and the map
$X = \tilde \psi \pa_\ell : V_1^- \to V_0^ -$ is still well defined and continuous.

The point~\eqref{item.prop.S.h1} follows from Lemma~\ref{lemma.local.alpha}. Let us now
prove the point~\eqref{item.prop.S.h2}.
To simplify the notation, we will assume that $P = \pa_i( a_{ij} \pa_j)$,
$a_{ij} \in \check W^{k+1,\infty}(U_0)$ (just one non-zero coefficient). Let then
\begin{equation*}
   g_t(x) \ede t^{-1} (a_{ij}(x+t \tilde \psi  e_\ell)-a_{ij}(x) ) - \tilde \psi  \partial_\ell (a_{ij}(x))\,. 
\end{equation*}
Then the $g_t \in W^{k,\infty}$ are uniformly bounded by
$2^k(1 + \| \tilde \psi \|_{W^{k+1,\infty}})  \| a_{ij} \|_{\check W^{k+1,\infty}}$
and therefore $\lim_{t \to 0} g_t w = 0$ in $\check H^{k}(U_0)$ for all
$w \in \check H^{k}(U_0)$. By taking $w = \partial_j u$
for all $j$, we obtain~\eqref{item.prop.S.h2}. 
\end{proof}

Now we state the corollary below which gives a version of Nirenberg's trick using vector fields.

\begin{corollary}\label{coro.esti.b.var} 
We use the notation of Proposition~\ref{prop.S}.
Suppose that $A \in M_{m+1}(\check{W}^{k+1,\infty}(U_0))$ and that
$ \maP_k = \maP_k^A : V_{k} \to V_{k}^-$ is bijective. Let $X(\maP_k^A) := \maP_k^{X(A)}$,
as in item~\ref{item.prop.S.h2} of that Proposition.
Then, for all $F \in V_{k+1}^-$ we have
\begin{equation*}
    X ( \maP_k^{-1} F ) \seq  \maP_k^{-1}(X F) - 
     \maP^{-1} X( \maP_k)  \maP_k^{-1} F \,.
\end{equation*}  
\end{corollary}

The corollary give then that 
\begin{equation}
     \Vert X (\maP_k^{-1} F) \Vert_{\check{H}^{k+1}} \seq 
     \Vert X (\maP_k^{-1} F) \Vert_{V_k} 
     \leq \|  \maP_k^{-1} X F  \Vert_{V_k} + 
     \|\maP^{-1} X ( \maP_k)  \maP_k^{-1} F \Vert_{V_k} \,.
\end{equation}

\begin{proof}
This is a direct and immediate consequence of Lemma~\ref{astuceniren} and
of Proposition~\ref{prop.S}. Indeed, let's take $\maX = V_{k}$, $\maY = V_{k}^-$,
$T := \maP_k = \maP_k^A$ and the groups of diffeomorphisms $S$ and $S_-$. It was assumed that
$T := \maP_k := \maP_k^A$ is bijective. Proposition
\ref{prop.S}~\eqref{item.prop.S.h1} and~\eqref{item.prop.S.h2}
shows that the other two hypotheses of Lemma
\ref{astuceniren} are satisfied with $Q:= X (\maP_k) = \maP_k^{X(A)}$. 
That lemma then gives that, for all $F$ in
$V_{k+1}^- \subset  D(L_{S_-})$, we have  $X (  \maP_k^{-1} F ) 
 =  \maP_k^{-1}(X f) -  \maP_k^{-1} X ( \maP_k)  \maP_k^{-1} F$, therefore
$\Vert X ( \maP_k^{-1} F) \Vert_{V_k} \leq \|  \maP_k^{-1} X F \Vert_{V_k} + 
\| \maP_k^{-1} X (  \maP_k)  \maP_k^{-1} F \Vert_{V_k} \,.$
\end{proof}

\section{Uniform estimates for families of operators}\label{sec4}
We now prove our main result, Theorem~\ref{thm.cons.transmission} and provide some applications. We consider
the setting already introduced. In particular, recall 
that all our domains $U_j$, $j =0, \ldots, N$, are bounded with smooth boundary and that $\pa_D U_0$,
$\pa_N U_0$ and the interface $\Gamma := \left (\cup_{j=1}^N \pa U_j \right ) \smallsetminus \pa U_0$
are compact, smooth, disjoint submanifolds of $\RR^m$. Also, recall that the 
sets $U_j$, $j = 1, \ldots, N$ are open and that $U_0 = \cup_{j=1}^N U_j \cup \Gamma$
is a disjoint union.

See Sections \ref{ssec.2.2},~\ref{ssec.2.3}, and~\ref{ssec.Pk} for notation and assumptions.

\subsection{Estimates for the norm of $(\maP_k^A)^{-1}$}
Theorem~\ref{thm.cons.transmission} is essentially equivalent to the following theorem,
for whose proof we shall repeatedly use Lemma~\ref{lemma.product.ge2}, usually without further comment.
Also, recall that we write $\|\xi \|_E = \infty$ if $\xi \notin E$, where $\|\cdot\|_E$ is the 
norm on $E$. 

In the following theorem, we are assuming, as usual, that $U_j$ are as in \ref{assum.decomposition}.
In that theorem, we are also using the resulting spaces 
$\check W^{k,\infty}(U_0)$ and $\check H^{k}(U_0)$ introduced in Definition 
\ref{def.broken.SS}, the spaces $V_n$ and $V_n^-$ introduced in Definition \ref{def.broken},
and the operators $\maP_k^A$ introduced in Definition \ref{def.opsP.k}.

\begin{theorem}\label{thm.transmission}
Let $k, n \ge 0$, let $A \in {M_{m+1}(\check W^{n+k+1,\infty}(U_0))}$, and suppose that
the induced operator $\maP_n = \maP_n^A : V_n \to V_n^{-}$ is invertible, where the 
notation was recalled in the paragraph above.
Given $F \in V_{n+k+1}^{-}$ we have
\begin{align*}
     \| \maP_n^{-1} F \| _{V_{n+k+1}} \lesssim \sum_{q=0}^{k+1} \, 
      |||  \maP_n^{-1} |||_{n}^{q+1} \,
      \nnW{n +k}{a_{mm}^{-1}}^{(q+1)k+1} 
      \|A\|_{\check W^{n+k+1,\infty}}^{(q+1)(k+1)} 
      \| F \|_{V_{n+k+1-q}^{-} } \,. \quad (\mathcal{J}_k)
\end{align*} 
\end {theorem}

Typically, we shall have $n = 0$, in which case,
we recall that $||| \maP_0^{-1} |||_{0}$ is the norm of the operator
$ \maP_0^{-1} = ( \maP_0^A)^{-1} : V_0^{-} := H_D^{1}(\RR^m)^* \to V_0 := 
H_D^1(U_0)$.

\begin{proof}
The proof is by induction on $k$ as in~\cite{Mirela1, theseHassan, HassanSimonVictor1}. 
The proof is the same for all $n$, so we shall assume $n = 0$, for simplicity.
\smallskip

\emph{Step 0: initial verification.}\
The estimate $(\mathcal{J}_{-1})$ is true by the definition of the norm
$ |||  \maP_0^{-1} |||_0 := |||  \maP_0^{-1} |||_{\maL(V_0^-,V_0)}$. Indeed, the relation
$(\mathcal{J}_{-1})$ simply reads
\begin{align*}
\|  \maP_0^{-1} F \|_{V_0} \le  |||  \maP_0^{-1} |||_0 \,  \|F \|_{V_0^{-}} \,.
\end{align*}

\smallskip
Let us proceed now to the \emph{induction property,} that is, 
prove the relation $(\mathcal{J}_{k+1})$ assuming the relations $(\mathcal{J}_s)$ for $s= -1, 0, 1, 2, \dots, k$.
The proof of the induction property will consist of several steps.
\smallskip

\emph{Step 1: Estimate for $\|\maP_0^{-1} F\|_{V_{k+1}}$.}
The induction hypothesis for $s = k$, i.e. for $F \in V_{k+1}^- \subset V_{k+1}^-$, 
says that the relation $(\mathcal{J}_k)$ is true for $F$, so
\begin{align} 
       \|  \maP_0^{-1} F\|_{V_{k+1}} 
       & \lesssim \sum_{q=0}^{k+1} \, 
      |||  \maP_0^{-1} |||_{0}^{q+1} \,
      \nnW{k}{a_{mm}^{-1}}
      ^{(q+1)k+1} \|A\|_{\check W^{k+1,\infty}}^{(q+1)(k+1)} 
      \| F \|_{V_{k+1-q}^{-}} \notag \\
      & \seq \sum_{q=0}^{k+1} \, |||  \maP_0^{-1} |||_{0}^{q+1} \,
%
      \nnW{k}{a_{mm}^{-1}}^{(q+1)(k+1) - q} \|A\|_{\check W^{k+1,\infty}}^{(q+1)(k+2) -1 - q} 
      \| F \|_{V_{k+1-q}^{-}} 
\notag \\
\label{eq.zero} 
      & \lesssim \sum_{q=0}^{k+1} \, 
      |||  \maP_0^{-1} |||_{0}^{q+1} \,
      \nnW{k}{a_{mm}^{-1}}
      \|A\|_{\check W^{k+1,\infty}}^{(q+1)(k+2) -1} 
      \| F \|_{V_{k+1-q}^{-}}      
      \, ,
\end{align}
where, for the last step, we also used Lemma~\ref{lemma.product.ge2}.
\smallskip

\emph{Step 2: Estimate for $\| X_\ell^\alpha ( \maP_0^{-1} F) \| _{V_{k+1}},$ 
$\ell < m$ for $F \in V_{k+2}^-$.}\ First, since $F \in V_{k+2}^-$ and $\ell < m$,
Proposition~\ref{prop.S}(3) gives us that, for all $\alpha$, we have
$X_\ell^\alpha F \in V_{k+1}^{-}$. So Corollary~\ref{coro.esti.b.var} 
for $ \maP_k =  \maP_{k+1}^A : V_{k+1} \to V_{k+1}^{-}$ and for the group of 
diffeomorphisms of generator $X_\ell^\alpha$ gives
\begin{align}\label{eqinitial.btrans.bis}
    \| X_\ell^\alpha ( \maP_0^{-1} F) \| _{V_{k+1}} \leq   
    \| \maP_0^{-1} (X_\ell^\alpha F)\|_{V_{k+1}} 
    + \|Q F \|_{V_{k+1}} \,,
\end{align}
with $Q :=  \maP_0^{-1} X_\ell^\alpha( \maP_0)  \maP_0^{-1} :V_{k+1}^{-} \rightarrow V_{k+1}$. 
We will now estimate the last two terms in the last equation.

\smallskip
First, Equation~\eqref{eq.zero} for $F$ replaced by $X_\ell^\alpha F
\in V_{k+1}^-$ gives:
\begin{align}\label{eqI.btrans.bis}
       \|  \maP_0^{-1} (X_\ell^\alpha F )\|_{V_{k+1}} 
       &  \lesssim
  \sum_{q=0}^{k+1} \, 
      |||  \maP_0^{-1} |||_{0}^{q+1} \,
%
      \nnW{k}{a_{mm}^{-1}}^{(q+1)(k+1)} \|A\|_{\check W^{k+1,\infty}}^{(q+1)(k+2) -1} 
      \| F \|_{V_{k+2-q}^{-}}      
      \, .
\end{align}

On the other hand, for $G := X_\ell^\alpha ( \maP_0) \maP_0^{-1} F \in V_{k+1}^{-} $,
the recurrence hypothesis $(\mathcal{J}_k)$ gives us, again, that
\begin{align}\label{eqg1.btrans.bis}
   \| Q F \|_{V_{k+1}} \ede \|  \maP_0^{-1}(G) \|_{V_{k+1}}  \lesssim \sum_{q=0}^{k+1} \, 
      |||  \maP_0^{-1} |||_{0}^{q+1} \,
%
      \nnW{k}{a_{mm}^{-1}}^{(q+1)k+1} \|A\|_{\check W^{k+1,\infty}}^{(q+1)(k+1)} 
      \| G \|_{V_{k+1-q}^{-}} \,.
\end{align}
In addition for $q$ fixed in $\{0, 1, \ldots , k+1\}$ we have, still by the hypothesis of
induction, that
\begin{align}\label{eqg2.btrans.bis}
    & \|G\|_{V_{k+1-q}^{-}}  \ede \| X_\ell^\alpha ( \maP_0)  \maP_0^{-1} F \|_{V_{k+1-q}^{-}} \leq 
    \|X_\ell^\alpha(A) )\|_{\check W^{k+1-q, \infty}} \| \maP_0^{-1} (F)\|_{V_{k+1-q}} \notag \\
    & \lesssim 
     \sum_{s=0}^{k-q+1} \,  |||  \maP_0^{-1} |||_{0}^{s+1} \,   
%
     \nnW{k-q}{a_{mm}^{-1}}^{(s+1)(k-q)+1} 
    \|A\|_{\check W^{k+2-q,\infty }}^{(s+1)(k-q+1)+1} \|F\|_{V_{k+1-q-s}^{-}} \,.
\end{align}
Let $t := q+ s + 1$ and $c(q,s) := sq + 2q + s \ge 0$
(because $s, q \ge 0$). Subsequently, using
the equations~\eqref{eqg1.btrans.bis} and~\eqref{eqg2.btrans.bis} 
as well as Lemma~\ref{lemma.product.ge2}, we have:
\begin{align}
    \| Q F\|_{V_{k+1}}  & \lesssim
    \sum_{q=0}^{k+1} \sum_{s=0}^{k+1-q} \, |||  \maP_0^{-1} |||_{0}^{t+1} \,  
%
    \nnW{k}{a_{mm}^{-1}}^{(t+1)(k+1) - c(q,s)} 
    \|A\|_{\check W^{k+2,\infty }}^{(t+1)(k+2) - c(q,s) - 1}  \|F\|_{V_{k+1-q-s}^{-}}
    \notag    
    \\
 & 
\label{eqII.btrans.bis} 
 \lesssim
     \sum_{t=1}^{k+2} \ \, |||  \maP_0^{-1} |||_{0}^{t+1} \,  
%
    \nnW{k}{a_{mm}^{-1}}^{(t+1)(k+1)} 
    \|A\|_{\check W^{k+2,\infty }}^{(t+1)(k+2)- 1}  \|F\|_{V_{k+2-t}^{-}}       
     \, .  
\end{align}

Finally, by replacing Equations~\eqref{eqI.btrans.bis} and~\eqref{eqII.btrans.bis} in
\eqref{eqinitial.btrans.bis}, for $\ell < m$, we find that:
\begin{align}\label{eq.one}
    \| X_\ell^\alpha ( \maP_0^{-1} F) \| _{V_{k+1}}  & \lesssim \sum_{q=0}^{k+2} \, 
    |||  \maP_0^{-1} |||_{0}^{q+1} \, 
%
    \nnW{k}{a_{mm}^{-1}}^{(q+1)(k+1)} 
    \|A\|_{\check W^{k+2, \infty}}^{(q+1)(k+2)-1}  \| F \|_{V_{k+2-q}^{-}}   \,.  
\end{align}                            

\emph{Step 3: Induction property estimate.}\                          
We now finally turn to the proof of the induction property.
So, let us assume that relation $(\mathcal{J}_s)$ is true for
$-1 \le s \le k$ and prove $(\mathcal{J}_{k+1})$ if 
$F \in V_{k+2}^{-}$ and $A \in M_{m+1}(\check W^{k+2,\infty}(U_0))$.

First, it follows from Lemma~\ref{lemma.product.ge2} that:
\begin{align}\label{eq.two}
\nnW{k+1}{a_{mm}^{-1}} 
 \| f \|_{V_{k+2}^{-}} \le    
 |||  \maP_0^{-1} |||_{0} \,  
\nnW{k+1}{a_{mm}^{-1}}\, 
\|A\|_{\check W^{k+2,}}^{k+1}
 \| F \|_{V_{k+2}^{-}} \,.
\end{align}
Let $F = (f, g, h)$ and $u :=  \maP_0^{-1}F$. Thus $Pu = f$. Next,
Lemma~\ref{lemma.sobolev1et2} (for $u = \maP_0^{-1}F$ and $k$ replaced
by $k+1$) and Equations~\eqref{eq.two},~\eqref{eq.zero},
and~\eqref{eq.one} imply (recall that for all $\alpha$, we set
$X_0^\alpha = id$):
\begin{align*}
   \| \maP_0^{-1}F\|_{\check H^{k+3}} \ & \lesssim 
%
   \nnW{k+1}{a_{mm}^{-1}}\Big ( \|f \|_{\check H^{k+1}} +
   \|A\|_{\check W^{k+2,\infty}} \sum_{\alpha \in I}\sum_{\ell=0}^{m-1} \|X_\ell^\alpha (u) \|_{\check H^{k+2}} \Big)
   \\
   & \lesssim \ |||  \maP_0^{-1} |||_{0} \,  
   \nnW{k+1}{a_{mm}^{-1}}
   ^{k+1} \, 
   \|A\|_{\check W^{k+2,\infty}}^{k+1} \| F \|_{V_{k+2}^{-}}  \\
  & + \ \sum_{q=0}^{k+1} \, 
      |||  \maP_0^{-1} |||_{0}^{q+1} \,
%
      \nnW{k+1}{a_{mm}^{-1}}^{(q+1)(k+1) +1} \|A\|_{\check W^{k+1,\infty}}^{(q+1)(k+2)} 
      \| F \|_{V_{k+1-q}^{-}} \\
  & + \
  \sum_{q=0}^{k+2} \, 
    |||  \maP_0^{-1} |||_{0}^{q+1} \, 
%
    \nnW{k+1}{a_{mm}^{-1}}^{(q+1)(k+1) +1} 
    \|A\|_{\check W^{k+2,\infty}}^{(q+1)(k+2)}  \| F \|_{V_{k+2-q}^{-}} \\  
  & \lesssim \ \sum_{q=0}^{k+2} \, 
      |||  \maP_0^{-1} |||_{0}^{q+1} \,
%
      \nnW{k+1}{a_{mm}^{-1}}^{(q+1)(k+1)+1} 
      \|A\|_{\check W^{k+2, \infty}}
      ^{(q+1)(k+2)} 
      \| F \|_{V_{k+2-q}^{-} } \,.
\end{align*}
The induction step is thus verified. This completes the proof of the relation
$(\mathcal{J}_{k+1})$. 
\end{proof}

\begin{remark} \label{rem.invert}
   The invertibility of $\maP_0 = \maP_0^A$ can be obtained using positivity
   (more precisely, the strong ellipticity) of suitable operators. This will be
   used in the next section. A generalization of the strong ellipticity is 
   ``$T$-coercivity,'' see~\cite{Ciarlet12, Ciarlet14, Karim1, Karim2,  Renata, Ciarlet13},
   which also may yield the invertibility of $\maP_0$. In these cases, we have
   $n = 0$ in our theorem. Yet another method of proving the invertibility of 
   $\maP_n$, for $n = 1$, is to use the self-adjointness of our operator (when this 
   the case), as in the last section of this paper, Section~\ref{sec.self-adjoint}.
\end{remark}

The main result stated in the introduction, Theorem~\ref{theorem.mainI},
follows right away from Theorem~\ref{thm.transmission}, as explained next:

\begin{proof}[Proof of Theorem~\ref{theorem.mainI}] 
   Since all boundaries $\pa U_j$ are smooth (so also the interface $\Gamma$ is smooth)
   and since $U_0$ is bounded, we obtain that there exists an extension constant $C_{U_0}$
   such that, for all $\tilde g \in H^{k+3/2}(\pa_D U_0)$ and all $\tilde h \in H^{k+3/2}(\Gamma)$,
   there exists $u_0 \in \check H^{k+2}(U_0)$ satisfying 
   $u_0\vert_{\pa_D U_0} = \tilde g$, $[[u_0]]_{\Gamma} = \tilde h$,
   and $$\|u_0\|_{\check H^{k+2}} \le C_{U_0} (\|\tilde g\|_{H^{k+1/2}(\pa_D U_0)} + 
   \|\tilde h\|_{H^{k+3/2}(\Gamma)})\,.$$
   Our last theorem (Theorem~\ref{thm.transmission}) applied to $F := (f - P^Au_0, g - D_\nu^A u_0, h - [[D_\nu^A u_0]])$
   then yields immediately Theorem~\ref{theorem.mainI}.
\end{proof}

\subsection{Automatic invertibility}
This subsection is devoted to some consequences of the following corollary 
that is based on Theorem~\ref{thm.transmission}, but does not use its full force.
We use the notation recalled before the statement of Theorem \ref{thm.transmission}.

\begin{corollary}\label{cor.thm.transmission}
   Let $k, n \ge 0$, let $A \in {M_{m+1}(\check W^{n+k,\infty}(U_0))}$, and suppose that
   the operator $\maP_n^A : V_n \to V_n^{-}$ is invertible.
   Then the operator $\maP_{n+k}^A : V_{n+k} \to V_{n+k}^-$  induced by
   restriction is also invertible. 
\end {corollary}

\begin{proof} We replace $k+1$ with $k$ in Theorem~\ref{thm.transmission}, in what
   follows, for convenience. We have that $\maP_{n+k}^A : V_{n+k} \to V_{n+k}^-$ is well defined 
   and continuous by Lemma~\ref{lemma.cont} since $A \in {M_{m+1}(\check W^{n+k,\infty}(U_0))}$.
   The hypothesis that $\maP_n$ is injective implies that its restriction 
   $\maP_{n+k}^A$ is also injective. Finally, Theorem~\ref{thm.transmission} implies that it is also
   surjective, and hence an isomorphism, as claimed.
\end{proof}

Our results also imply a regularity result.

\begin{remark}\label{rem.regularity}
   Under the assumptions of Theorem~\ref{thm.transmission} (especially $F \in V_{n+k}^{-}$ and
   $A$ with coefficients in $\check W^{n+k,\infty}(U_0)$), we obtain, in particular, that $\maP_n^{-1} F \in V_{n+k}$.
   This is a \emph{regularity result.} In the strongly elliptic case, the oldest statement we know of such a
   regularity result is in the work of Roitberg and Sheftel~\cite{RoitbergSh62, RoitbergSh63}. 
   See also~\cite{LiNistorQiao}. See~\cite{HMN, NicaiseBook, NicaiseSandig1, NicaiseSandig2} for 
   the case of polygons, where, we stress, the results may be very different.
\end{remark}

If $E$ and $F$ are two normed spaces, recall that $\maL(E; F)$ denotes the set of continuous,
linear maps $E \to F$. Moreover, we shall write $\maL(E; F)^{-1}$ for the set of 
continuous, \emph{invertible} linear maps $E \to F$. 
We shall need also the well known fact that the map $\maL(E; F)^{-1} \ni T \to T^{-1} \in \maL(F; E)$ 
is analytic (hence continuous, hence measurable). This gives the following consequences.

\begin{corollary}\label{cor.mca}
   Let $n, k \in \ZZ_+$ and $A : \Theta \to M_{(m+1)}(\check W^{n+k, \infty}(U_0))$. 
   We assume that, for all $\theta \in \Theta$, $\maP_n^{A(\theta)} : V_n \to V_n^{-}$ is 
   invertible, and hence that the function $$\big (\maP_{n+k}^{A} \big)^{-1} : \Theta  \to \maL(V_{n+k}^-;V_{n+k})\,,$$
   $(\maP_{n+k}^{A})^{-1}(\theta) := \big (\maP_{n+k}^{A(\theta)} \big )^{-1}$, is 
   well-defined and we have the following:
   \begin{enumerate}[(i)]
   \item If $\Theta$ is a measurable space and $A$ is measurable, then $\big (\maP_{n+k}^{A} \big)^{-1}$
   is measurable.

   \item If $\Theta$ is a topological space and $A$ is continuous, then $\big (\maP_{n+k}^{A} \big)^{-1}$
   is continous.

   \item If $\Theta$ is an open subspace in a locally convex, topological vector
   space and $A$ is analytic, then $\big (\maP_{n+k}^{A} \big)^{-1}$
   is analytic.
   \end{enumerate}
\end{corollary}

\begin{proof} Let us prove (i), since the other two points are completely similar.
   Since $M_{(m+1)}(\check W^{n+k, \infty}(U_0)) \ni B \to \maP_{n+k}^B 
   \in \maL(V_{n+k}^-;V_{n+k})$ is continuous and linear (Lemma~\ref{lemma.cont}), it follows that the 
   map $\maP_{n+k}^{A} : \Theta \to \maL(V_{n+k}^-;V_{n+k})$ is also measurable. By Corollary~\ref{cor.thm.transmission}, 
   we know that $\maP_{n+k}^{A} : \Theta \to \maL(V_{n+k}^-;V_{n+k})$
   has values invertible elements. Since $T \to T^{-1}$
   is measurable (even analytic!), the result follows from the fact that the composition of two measurable 
   maps is measurable. 
\end{proof}

\subsection{Integrability from uniform boundedness}
We now include a result that uses the full force of Theorem~\ref{thm.transmission}.
The main notation was recalled before the statement of Theorem \ref{thm.transmission},
and we continue to use it.
First we refine the invertibility result of Corollary~\ref{cor.thm.transmission} as 
follows:

\begin{corollary}\label{cor.thm.transmission2}
   Let $A \in {M_{m+1}(\check W^{n+k+1,\infty}(U_0))}$, $k, n \ge 0$, $\maP_k := \maP_k^A$,
   and suppose that the operator $\maP_n : V_n \to V_n^{-}$ is invertible.
   Then $\maP_{n+k+1} : V_{n+k+1} \to V_{n+k+1}^-$ is also invertible and its inverse has norm
   $$|||\maP_{n+k+1}^{-1}|||_{n+k+1} \, \lesssim \, |||\maP_n^{-1} |||_{n}^{k+2} \,
           \nnW{n+k}{a_{mm}^{-1}}^{(k+1)^2} 
         \|A\|_{\check W^{n+k+1,\infty}}^{(k+2)(k+1)}\,.$$
\end {corollary}

\begin{proof}
   Lemma~\ref{lemma.product.ge2}
   implies that $|||  \maP_n^{-1} |||_{n}^{q+1} \,
   \nnW{n +k}{a_{mm}^{-1}}^{(q+1)k+1} 
   \|A\|_{\check W^{n+k+1,\infty}}^{(q+1)(k+1)}$ is increasing in $q$. Since $\| F \|_{V_{k+2-q}^{-} }$
   is decreasing in $q$, the result follows.
\end{proof}

The above result is obviously true also for $k = -1$, as long as one defines $\nnW{n+k}{a_{mm}^{-1}}$
for $n = 0$ (actually, one can simply ignore that term for $n = 0$ and $k = -1$). In order to avoid
this discussion, it was convenient to state the above result (as well as Theorem~\ref{thm.transmission})
in the given range of $k$ (i.e. $k \ge 0$). We now shift back and replace $k+1$ with $k$.
We obtain the following consequence.

\begin{proposition}\label{prop.bounded}
   Let $n, k \in \ZZ_+$. Suppose that 
   \begin{enumerate}[(i)]
   \item $A : \Theta \to M_{(m+1)}(\check W^{n+k, \infty}(U_0))$
   is bounded;
   
   \item $\maP_n^{A(\theta)} : V_{n} \to V_{n}^-$ is invertible for
   all $\theta \in \Theta$; and
   
   \item the function $\theta \to \big (\maP_n^{A(\theta)} \big )^{-1} \in \maL(V_{n}^- ; V_{n})$
   is bounded. 
   \end{enumerate}
   Then the function $\big (\maP_{n+k}^{A} \big )^{-1}(\theta) := \big (\maP_{n+k}^{A(\theta)} \big ) ^{-1} \in \maL(V_{n+k}^- ; V_{n+k})$
   is also bounded. Consequently, if $\Theta$ is a probability space and $A$ is 
   measurable, then the functions $\big (\maP_{n+k}^{A} \big )^{-1}$ and $\| \big (\maP_{n+k}^{A} \big )^{-1}\|$
   are integrable on $\Theta$.
\end{proposition}

\begin{proof} The first part follows right away from Corollary~\ref{cor.thm.transmission2}.
   The second part follows by combining the first part with Corollary~\ref{cor.mca}(i), since
   bounded, measurable functions on probability spaces are integrable.
\end{proof}

The main applications in this paper will be to extend the above integrability result to 
the case when the norms of the operators $\maP_n^{A(\theta)} \in \maL( V_{n}; V_{n}^-)$
are not uniformly bounded.

\section{Estimates for strongly elliptic operators and integrability}\label{sec5}

In this section, we include some further applications, most notably, an integrability
result that extends Proposition~\ref{prop.bounded} to the case when 
the norms of the operators $\maP_n^{A(\theta)} \in \maL( V_{n}; V_{n}^-)$
are not uniformly bounded. We use the same notation
and assumptions as in the last sections. In particular, $U_0 \subset \RR^m$ 
is bounded with smooth boundary and is endowed with a decomposition into subdomains $U_j$
along a smooth interface $\Gamma := (\bigcup_{k=1}^N \pa U_k) \smallsetminus \pa U_0$. 
We also allow the case $U_k = (a_k, b_k) \times \RR^{m-1}$,
for which the results of the previous sections are replaced with the results 
in Chapters 2 and 3 of~\cite{theseHassan},
or with the results in~\cite{HassanSimonVictor1}.

\begin{notation}\label{not.sec5}
See Sections \ref{ssec.2.2},~\ref{ssec.2.3}, and~\ref{ssec.Pk} for notation and assumptions.
More precisely, the sets $U_j$ are as in \ref{assum.decomposition}, the resulting spaces 
$\check W^{k,\infty}(U_0)$ and $\check H^{k}(U_0)$ are as in Definition 
\ref{def.broken.SS}, the spaces $V_n$ and $V_n^-$ are as in Definition \ref{def.broken},
and the operators $\maP_k^A$ are as in Definition \ref{def.opsP.k}. In particular,
$\overline{U}_0 = \cup_{j=1}^N U_j \cup \Gamma \cup \pa_D U_0 \cup \pa_N U_0$ is a
disjoint union with $U_j$ all open and $\Gamma$, $\pa_D U_0$, and $\pa_N U_0$
compact, smooth submanifolds of $\RR^m$. All the necessary notation was also introduced
in the Introduction.
\end{notation}

\subsection{General case}
We will need the following standard lemma (see, for example, 
Chapter 5 of~\cite{theseHassan} or one of the following papers 
\cite{Victor, Mirela1, HassanSimonVictor1}).

\begin{lemma}\label{lemma.inverse} Let $b \in \check W^{k,\infty}(U_0)$ 
   (see Definition   \ref{def.broken.SS}) be such that 
   $b^{-1} \in L^{\infty}(U_0)$. Then $b^{-1} \in \check W^{k,\infty}(U_0)$ and
   \begin{align*}
      \|b^{-1} \|_{\check W^{k,\infty}} \leq c_3 \|b^{-1} \|_{L^{\infty}}^{k+1} 
      \|b \|_{\check W^{k,\infty}}^k \,.
   \end{align*}
\end{lemma}

We will denote $\RE A := \frac12 \big (A + A^* \big)$, for any
matrix $A \in M_{m+1}(\CC)$ (with $A^*$ the adjoint of $A$, i.e. its
transposed conjugate). We write $A \ge \gamma I_{m+1}$ if, for
any complex vector $\xi \in \CC^{m+1}$ on which $A$ acts,
we have $(A\xi, \xi) \ge \gamma \|\xi \|^2$. (Thus, $I_{m+1}$ is the identity 
matrix of $M_{m+1}(\CC)$.) We write $A > 0$ if
$A \ge 0$ and $A$ is invertible. If $A \in M_{m+1}(L^\infty(U_0))$,
the inequality $A \ge \gamma I_{m+1}$ means that $A(x) \ge \gamma I_{m+1}$
for almost all $x \in U_0$. In the following, we will also omit $A$ in the notation 
of the operators $P^A$ and $\maP_k^A$ when there is no risk of confusion.

\begin{lemma} \label{lemma.majoration}
   See Notation \ref{not.sec5} for a review of the notation. Let $A \in M_{m+1}(L^\infty(U_0))$.
   If $\RE A \ge \gamma I_{m+1}$, then $\maP_0 = \maP_0^A : V_0 \to V_0^-$ is invertible
   with $|||\maP_0^{-1}|||_0 \le \gamma^{-1}$.
\end{lemma}

\begin{proof}
Let $\maP_0 := \maP_0^A$. As $\gamma > 0$, we obtain for $u \in V_0 = H^1(U_0)$,
\begin{align*}
   \RE (\maP_0 u, u) \ede \RE B^A(u, u) &\seq \RE \int_{U_0} 
   \sum_{i,j = 0}^{m} a_{ij} \pa_i u \pa_j \overline{u}\, dx \\
   & \ge \gamma \int_{U_0} \sum_{i=0}^{m} |\pa_i u|^2 \, dx
   \, =:\, \gamma \|u\|_{H^1}^2 \,.
\end{align*}
So $\maP_0$ is invertible and $||| \maP_0^{-1}|||_{0} \le \gamma^{-1}$,
by Lax--Milgram's lemma. 
\end{proof}

\begin{theorem}\label{thm.cons.transmission}
Suppose that $U_0 = \bigcup_{k=1}^N U_k \subset \RR^m$ is bounded with $\pa U_k$
smooth, as usual. Let $A \in M_{m+1}(\check W^{k+1,\infty}(U_0))$ be such that 
$A \ge \gamma I_{m+1}$, with $\gamma >0$. Let $\maP_j := \maP_j^A : V_j \to V_j^-$ 
for all $j$ and $k \ge -1$. (See Notation \ref{not.sec5} for a review of the notation.)
Then, given $F \in V_{k+1}^{-}$, we have
\begin{align*}
     \|\maP_0^{-1} F \| _{V_{k+1}} \lesssim \sum_{q=0}^{k+1} \, 
      \gamma^{-q-1} \,
%
      \nnW{k}{a_{mm}^{-1}}^{(q+1)k+1} 
      \|A\|_{\check W^{k+1,\infty}}^{(q+1)(k+1)} 
      \| F\|_{V_{k+1-q}^{-} } \,. 
\end{align*} 
In particular, $\maP_{k+1}$ is invertible. Let $K := (k+2)(k^2 + k + 1) + k$. Then
$$||| \maP_{k+1}^{-1} |||_{k+1} \le \gamma^{-K-1} \, 
\|A\|_{\check W^{k+1,\infty}}^{K} \,.$$
\end{theorem}

\begin{proof} 
The hypothesis of the theorem imply that $\maP_0 = \maP_0^A : V_0 := H_D^1(U_0) \to
V_0^{-} =: H_D^{1}(U_0)^*$ is invertible and
$||| \maP_0^{-1}|||_{0} \le \gamma^{-1}$, according to Lemma~\ref{lemma.majoration}.
Moreover, we also have $\gamma \lesssim a_{mm}$, so $\|a_{mm}^{-1}\|_{L^\infty} \lesssim \gamma^{-1}$.
It follows from Lemma~\ref{lemma.inverse} that
\begin{equation*}
%
   \nnW{k}{a_{mm}^{-1}} \le \gamma^{-k-1}
%
   \nnW{k}{a_{mm}^{-1}}^k \le \gamma^{-k-1} \|A\|_{\check W^{k+1,\infty}}^k \,.
\end{equation*}
The rest is a consequence of Theorem~\ref{thm.transmission}.
\end{proof}

We obtain the following consequence for the case where the coefficients are
random Gaussian variables defined on a measured space $\Omega$.

%

\begin{theorem} \label{thm.integr}
Let $X =( X_1, X_2, \ldots, X_q): \Omega \to \RR^q$ be a Gaussian vector random variable
with covariance $\sigma = (\sigma_{ij}) > 0$ and let $\gamma > 0$. Let $A_1, A_2, \ldots, A_q
\in M_{m+1}(\check W^{k+1,\infty}(U_0))$, $\RE A_\ell \ge 0$ for all $1 \le \ell \le q$, such that, 
for all $x\in U_0$, there is $\ell$ with $\RE A_\ell(x) \ge \gamma I_{m+1}$. We let $A(\omega) 
:= \sum_{\ell=1}^q e^{X_\ell(\omega)} A_j$ and $\gamma(\omega)$ be the largest constant
such that $\gamma(\omega) I \le \RE A(\omega)$. Then, for all $0 < p, r, s < \infty$, the function
$$\Omega \ni \omega \to  \gamma(\omega)^{-s} \|A(\omega)\|_{\check W^{k+1}}^r
||| \big (\maP_{k+1}^{A(\omega)} \big )^{-1} |||_{k+1}^p \in (0, \infty)$$ is integrable.
\end{theorem}

\begin{proof} The given function is measurable by Corollary~\ref{cor.mca}(i).
We have 
\begin{equation}\label{eq.lb}
   \RE A(\omega) \ede  \RE \sum_{\ell=1}^q e^{X_\ell(\omega)} A_j  \ge \gamma\, 
   \min_{\ell=1}^q \{ e^{X_\ell(\omega)} \} \, I_{m+1} \,,
\end{equation} 
where $I_{m+1}$ is the identity matrix of $M_{m+1}(\CC)$, as before.
Hence $\gamma(\omega) \ge \gamma\, 
\min_{\ell=1}^q \{ e^{X_\ell(\omega)} \}$.
Similary, we have the estimate
\begin{equation}
   \|A(\omega)\|_{\check W^{k+1,\infty}} \, \le \,
   \sum_{j=1}^q e^{X_\ell(\omega)}  \| A_j \|_{\check W^{k+1,\infty}}  \,.
\end{equation}
By Theorem~\ref{thm.cons.transmission}, we obtain
\begin{align*}
   \gamma(\omega)^{-s/p} \|A(\omega)\|_{\check W^{k+1}}^{r/p} \, ||| (\maP_{k+1}^{A(\omega)})^{-1} |||_{k+1} 
   & \le \big( \gamma \min_{\ell=1}^q \{ e^{X_{\ell}(\omega)} \} \big)^{-K-s/p-1} \,
   \Big (\sum_{\ell=1}^q e^{X_\ell(\omega)} \|A_\ell\|_{\check W^{k,\infty}} \Big)^{K+r/p} \\
   & \le C \Big ( \sum_{\ell=1}^q e^{|X_\ell(\omega)|} \Big)^{2K + r/p + s/p + 1}\,.
\end{align*}
Hence
\begin{align*}
   \int_\Omega \gamma(\omega)^{-s/p}  \|A(\omega)\|_{\check W^{k+1}}^{r} \, ||| (\maP_{k+1}^{A(\omega)})^{-1} |||_{k+1}^p d\omega
   & \le \ C \, \int_\Omega \, \Big ( \sum_{\ell=1}^q e^{|X_\ell(\omega)|} \Big)^{2K + r/p + s/p + 1} d\omega \\
   & \seq \ C \, \int_{\RR^q} \, \Big ( \sum_{\ell=1}^q e^{|x_\ell|} \Big)^{a} e^{- (\sigma^{-1} x, x)} dx \, < \, \infty
\end{align*}
since, for all $a > 0$,
 $\big(\sum_{\ell=1}^q e^{|x_\ell|} \big)^{a} < (q + e^{|x_1| + 
\ldots + |x_q|})^a,$ which is integrable with respect to the measure of density 
$e^{-(\sigma^{-1} x, x)} \le e^{-\epsilon \|x\|^2}$ on $\RR^q$.
\end{proof}

\subsection{Poincaré inequality}
We shall say that the Poincar\'e inequality is satisfied by the functions in
$H_D^1(U_0) := \{ u \in H^1(U_0) \mid u\vert_{\pa_D U_0} = 0\}$ with
constant $\eta_{U}$ if, for all $u \in H_D^1(U_0)$, we have
\begin{equation}\label{eq.def.ineq.Poincare}
   |u|_{H^1(U_0)}^2 \ede \int_{U_0} |\nabla u(x)|dx \, \ge\, \eta_{U}^2 \|u\|_{H^1(U_0)}^2\,.
\end{equation}
If this condition is satisfied, then we get stronger results as follows.
Let $A = [a_{ij}] \in M_{m+1}(L^\infty(U_0))$ and $P^A$ and $\maP_k^A$ be the associated
operators, as before. Recall then that $A$ (or $P^A$, or, yet, $\maP_k^A$) 
is \emph{$\gamma$--strongly elliptic} if, for all $x \in U_0$ and all $\xi \in \CC^m$, we have
\begin{equation}\label{eq.gamma.se}
   \sum_{i,j=1}^m a_{ij}(x) \xi_i \overline{\xi}_j \ge \gamma |\xi|^2\,.
\end{equation}
The largest $\gamma$ with this property will be denoted $\gamma(A)$ and is called the 
\emph{coercivity constant} of $A$. Of course $\gamma(A) = \|(A + A^*)^{-1}\|^{-1}/2$ if $A$
is coercive.
Let $I' := 0 \oplus I_{m} \in M_{m+1}(\CC)$ be the projection with entries $(I')_{ij}$ given by 
\begin{equation}\label{eq.def.P0}
   \begin{cases} 
      (I')_{ii} = 1 & \mbox{ for } 1 \le i \le m \\
      (I')_{ij} = 0 & \mbox{ otherwise.}
   \end{cases}
\end{equation}
(Thus $I'$ differs from the identity matrix $I_{m+1} \in M_{m+1}(\CC)$ in only one entry.)
The strong ellipticity condition then becomes $I' A I' \ge \gamma I'$.
If $a_{0j} = a_{i0} = 0$ for all $0 \le i, j \le m$, 
then $I' A I' = A$ and the strong ellipticity condition becomes $A \ge \gamma I'$.
The interest of the concept of strong ellipticity is that operators satisfying this
condition have the following properties.

\begin{lemma} \label{lemma.majoration.FE}
Let $\gamma > 0$ and $A = [a_{ij}]\in M_{m+1}(L^\infty(U_0))$ be such that
$a_{i0} = a_{0j} = 0$ for all $1 \le i, j \le m$ and $\RE A \ge 0$. Suppose 
that $A \ge \gamma I'$ (i.e. $A$ is $\gamma$--strongly elliptic) and that
the Poincaré inequality is satisfied on $H_D^1(U_0)$ with constant $\eta_{U} > 0$.
Then
$\maP_0 := \maP_0^A: V_0 \to V_0^-$ is invertible with norm 
$|||\maP_0^{-1}|||_0 \le (\eta_{U} \gamma)^{-1}$.
\end{lemma}

\begin{proof} 
We have
\begin{equation*}
   \RE (\maP_0^A u, u) \seq B^A(u, u) \ge \gamma^2 |u|_{H^1(U_0)}^2 + a_{00} \|u\|_{L^2}
   \ge (\eta_{U} \gamma)^2 \|u\|_{H^1(U_0)}^2 \,,
\end{equation*}
then the Lax--Milgram's Lemma gives the results in view of the definitions of 
$V_0 := H_D^1(U_0)$ and $V_0^- := V_0^*$.
\end{proof}

We immediately obtain the following consequence, with essentially the same proof
as that of Theorem~\ref{thm.cons.transmission}. See Notation \ref{not.sec5} for 
a review of the notation.

\begin{theorem} \label{thm.cons.transmission2}
Suppose that $A \in \check W^{k+1,\infty}(U_0)$
satisfies the hypotheses of Lemma~\ref{lemma.majoration.FE}.
Let $\maP_j := \maP_j^A : V_j \to V_j^-$ for all $j$ and $k \ge -1$. Then, given $F \in V_{k+1}^{-}$ we have
\begin{align*}
     \|\maP_0^{-1} F \| _{V_{k+1}} \lesssim \sum_{q=0}^{k+1} \, 
      \gamma^{-q-1} \,
%
      \nnW{k}{a_{mm}^{-1}}^{(q+1)k+1} 
      \|A\|_{\check W^{k+1,\infty}}^{(q+1)(k+1)} 
      \| F\|_{V_{k+1-q}^{-} } \,. 
\end{align*} 
In particular, $\maP_{k+1}$ is invertible. Let $K := (k+2)(k^2 + k + 1) + k$. Then
$$||| \maP_{k+1}^{-1} |||_{k+1} \lesssim \gamma^{-K-1} \,
      \|A\|_{\check W^{k+1,\infty}}^{K} \,.$$
\end{theorem}

   
We obtain the following consequence (with the same proof as Theorem~\ref{thm.integr}).
We use the hypotheses of the Theorems~\ref{thm.integr} and~\ref{thm.cons.transmission2}
(listed below):

\begin{theorem} \label{thm.integr2} 
   We continue to use Notation \ref{not.sec5}. Let $\gamma > 0$ and 
   $A_1, A_2, \ldots, A_q \in M_{m+1}(\check W^{k+1,\infty}(U_0))$ be matrices such that 
   \begin{enumerate}[(i)]
      \item $(A_\ell)_{i0} = (A_\ell)_{0j} = 0$, for all $1 \le i, j \le m$ and all $1 \le \ell \le q$,
  
      \item $\RE A_\ell \ge 0$ for all $1 \le \ell \le q$, and 

      \item \label{item.P0}
      for each $x \in U_0$, there is $1 \le \ell \le q$ such that $\RE A_\ell(x) \ge \gamma I'$,
      where $I' := 0 \oplus I_{m}$ (see Equation~\eqref{eq.def.P0}).
   \end{enumerate}
   We assume that the Poincaré inequality is satisfied on $H_D^1(U_0)$. Let $X: \Omega \to \RR^q$
   be Gaussian as in Theorem~\ref{thm.integr} and $A(\omega) := \sum_{\ell=1}^q e^{X_\ell(\omega)} A_\ell$. 
   Then $A(\omega)$ is coercive with coercivity constant $\gamma(\omega) := \gamma(A(\omega))$ such that,
   for all $0 < p, r, s < \infty,$ the function
   $$\Omega \ni \omega \to \gamma(\omega)^{-s}\|A\|_{\check W^{k+1}}^r 
      ||| \big (\maP_{k+1}^{A(\omega)} \big )^{-1} |||_{k+1}^p \in (0, \infty)$$ is integrable.
\end{theorem}

The conditions $(A_\ell)_{i0} = (A_\ell)_{0j} = 0$, for all $0 \le i, j \le m$ and all $1 \le \ell \le q$
model homogeneous second order operators (such as ``sign-changing Laplacians''). Our operators 
are slightly more general since we allow $a_{00} \ge 0$.
They can be replaced with some less restrictive conditions, but then the results are 
more technical to state. The condition \eqref{item.P0} means, of course, that the operator associated to 
$A_\ell$ is strongly elliptic at $x$ and hence that the sum $A(\omega) := \sum_{\ell=1}^q e^{X_\ell(\omega)} A_\ell$
is strongly elliptic everywhere.

\begin{proof} Equation \eqref{eq.lb} of the proof of Theorem 
   \ref{thm.integr} becomes
   \begin{equation*}
      \RE A(\omega) \ede  \RE \sum_{\ell=1}^q e^{X_\ell(\omega)} A_j  \ge \gamma\, 
      \min_{\ell=1}^q \{ e^{X_\ell} \} \, I' \,.
   \end{equation*}
   We also have that $\RE A(\omega) \ge 0$ and that $(A(\omega))_{i0} = (A(\omega))_{0j}$
   for all $1 \le i, j \le m$. Thus we can use Theorem \ref{thm.cons.transmission2}
   (instead of Theorem \ref{thm.cons.transmission}) to complete the proof as in the 
   proof of Theorem \ref{thm.integr}.
\end{proof}

\subsection{Application to the Finite Element Method}
Let $S_N \subset H_D^1$ be a sequence of finite dimensional vector subspaces. 
We let $Q_N$ denote the projection of $u \in H_D^1$ on $S_N$ in the norm $\|\cdot \|_{H^1}$.
We suppose that there exist $k > 1$, $C_{rate} > 0$, and $\mu > 0$ such that, 
for all $u \in \check H^{k+1}(U_0)$, we have
\begin{equation}\label{eq.conv.rates}
   \| u - Q_N u \|_{H^1(U_0)} \le C_{rate} \dim(S_N)^{-\mu} \|u\|_{\check H^{k+1}(U_0)}\,.
\end{equation}
See~\cite{qingqin} for an example of such subspaces using the Generalized Finite
Element spaces and~\cite{BNZ3D2} for an example of such subspaces using anisotropically
graded meshes on 3D polyhedral domains. Let us denote by $Q_N^A$ the projection onto 
$S_N$ in the inner product $B^A$ whenever $A = A^*$ and $A$ satisfies the hypotheses of 
Lemma \ref{lemma.majoration.FE} (this guarantees then that $B^A$ defines an inner product
on $H_D^1(U_0)$).

\begin{proposition} \label{prop.comp.FEM}
   We keep the notation and 
   assumptions of Lemma \ref{lemma.majoration.FE} and of Equation \eqref{eq.conv.rates}. 
   Let $A = A^*$ and $Q_N^A$ be the projection onto $S_N$ in the inner product $B^A$. Then
   \begin{equation*}
      \|u - Q_N^A u\|_{H^1(U_0)} \le C_{rate} (\gamma \eta_U)^{-1/2} 
      \|A\|_{L^\infty}^{1/2} \dim(S_N)^{-\mu} \|u\|_{\check H^{k+1}(U_0)}\,.
   \end{equation*}
\end{proposition}

\begin{proof} Let $\dist_{(\, , \,)}$ be the distance in some inner product $(\, , \, )$
   (or in some norm induced from an inner product). 
   We have 
   $$\gamma \eta_U\|u\|_{H^1}^2 \le \gamma \|\nabla u\|_{L^2}^2 \le B^A(u, u) \le \|A\|_{L^\infty} \|u\|_{H^1}^2\,.$$ 
   Hence we have successively
   \begin{align*}
      \|u - Q_N^A u\|_{H^1(U_0)} &  \le (\gamma \eta_U)^{-1/2} 
      B^A(u - Q_N^A u, u - Q_N^A u)^{1/2} \\
      & \seq (\gamma \eta_U)^{-1/2}  \dist_{B^A}(u, S_N) \\
      & \le (\gamma \eta_U)^{-1/2} \|A\|_{L^\infty}^{1/2} \dist_{\|\,\cdot \,\|_{H^1}}(u, S_N)\\
      & \seq (\gamma \eta_U)^{-1/2} \|A\|_{L^\infty}^{1/2} \|u - Q_N u\|_{H^1(U_0)}\\
      & \leq C_{rate} (\gamma \eta_U)^{-1/2} \|A\|_{L^\infty}^{1/2} \dim(S_N)^{-\mu} \|u\|_{\check H^{k+1}(U_0)}\,,
   \end{align*}
where the last inequality is by the assumption \eqref{eq.conv.rates}.
\end{proof}

We then obtain the following theorem.

\begin{theorem} \label{thm.integr3}
   We use the assumptions and notation of Theorem~\ref{thm.integr2} and the 
   condition \eqref{eq.conv.rates}. Then, there exists $C_{X, p}$ such that, for all $F \in V_k^-$, we have 
   \begin{equation*}
      \int_{\Omega} \| (\maP_{k+1}^{A(\omega)})^{-1} F - 
      Q_N (\maP_{k+1}^{A(\omega)})^{-1} F \|_{H^1}^p d \omega
      \ \le \ C_{p,X} \dim(S_N)^{-p \mu} \|F\|_{V_k^-}^p\,.
   \end{equation*}
   If $A_\ell = A_\ell^*$ for all $\ell$, then there exists also
   $C_{X, p}^{FEM}$ such that, for all $F \in V_k^-$,
   \begin{equation*}     
      \int_{\Omega} \| (\maP_{k+1}^{A(\omega)})^{-1} F - Q_N^{A(\omega)} 
      (\maP_{k+1}^{A(\omega)})^{-1} F \|_{H^1}^p d \omega
      \ \le \ C_{p,X}^{FEM} \dim(S_N)^{-p \mu} \|F\|_{V_k^-}^p\,.
   \end{equation*}
\end{theorem}

\begin{proof}
Equation~\eqref{eq.conv.rates} for $u := (\maP_{k+1}^{A(\omega)})^{-1} F$
immediately gives:
\begin{align*}
   \| (\maP_{k+1}^{A(\omega)})^{-1} F - Q_N (\maP_{k+1}^{A(\omega)})^{-1} F \|_{H^1} 
       & \le C \dim(S_N)^{-\mu} \| (\maP_{k+1}^{A(\omega)})^{-1} F \|_{\check H^{k+1}} \\
       & \le C \dim(S_N)^{-\mu} ||| (\maP_{k+1}^{A(\omega)})^{-1} |||_{k} \|F\|_{V_k^-} \,.
\end{align*}
So we may just choose $C_{p,X} := C^p \int_{\Omega} ||| (\maP_{k+1}^{A(\omega)})^{-1}|||_{k}^p\, d \omega$,
which is finite, by Theorem~\ref{thm.integr2}. For the second estimate, we
proceed similarly, but using Proposition \ref{prop.comp.FEM} instead of Equation~\eqref{eq.conv.rates}
and 
$$C_{p,X}^{FEM} \ede C^p \int_{\Omega} \gamma(\omega)^{-p/2}\|A(\omega)\|_{\check W^{k+1}}^{p/2}  |||
(\maP_{k+1}^{A(\omega)})^{-1} F |||_{k}^p\, d \omega\,,$$
which is finite by the same theorem.
\end{proof}

\begin{remark} In all of the above results, we may choose
   $F = F(X_1, X_2, \ldots, X_q)$ to depend on $X_\ell(\omega)$ as well, $1 \le \ell \le q$, 
   as long as the norm growth of $F(x_1, \ldots, x_q)$ in $x$ is at most exponential. 
\end{remark}

\begin{remark} The last part of the last theorem gives an estimate for 
   the average error in the Finite Element Method approximation for all 
   the equations $\maP_0^{A(\omega)}u = F$, since $Q_N^{A(\omega)}(\maP_0^{A(\omega)})^{-1}F 
   \in S_n$ is the Finite Element approximation of the solution $u$ in the discretization
   space $S_N$.
\end{remark}

\section{Self-adjoint operators, boundary triples, and parametric regularity}
\label{sec.self-adjoint}
As discussed in Remark~\ref{rem.invert}, the invertibility of $\maP_0$ in our main result, Theorem~\ref{thm.transmission}, 
follows from strong ellipticity or from $T$-coercivity. In particular, 
we have seen that the hypotheses of Theorem~\ref{thm.transmission}
are satisfied if $A \ge \gamma I$ or if it is strongly elliptic
and the Poincaré's inequality is satisfied. In this section, we employ 
a different setting, namely that of self-adjoint operators, in which case the 
invertibility is obtained from their spectral properties. 
We continue to assume that $\Gamma$ and all boundaries $\pa U_k$, $k = 0, \ldots, N$, are smooth.
We will need the following theorem~\cite{Behrndt2014, Karim1, 
Ciarlet13, Pankrashkin16, Pankrashkin19}, which, we stress, is formulated only for Dirichlet
boundary conditions, so $\pa_N U_0 = \emptyset$ and $\pa U_0 = \pa_D U_0$.

\begin{theorem} \label{thm.self-adjoint} 
Let $A = a_j I$ on $U_j$ (scalar matrix), where $a_j \in \RR^*$ are constants
such that $a_j + a_i \neq 0$ if $U_i$ and $U_j$ are adjacent.
Then $P^A$ with domain 
$$\mathcal{D} \ede \check H^2(U_0) \cap H_0^1(U_0) \cap \{ [[D_{\nu}^{A} u]]_\Gamma = 0 \}$$
is self-adjoint. Therefore $\maP_k^A + t \imath :V_{k} \to V_{k}^-$, $0 \neq t \in \RR$, is invertible
for all $k \ge 1$.
\end{theorem}

So the assumptions of our main theorem, Theorem~\ref{thm.transmission} are also satisfied for 
$P + t \imath$, $t \neq 0$, $\imath := \sqrt{-1}$. By replacing $\maP_0^{-1}$ with 
$(\maP^{-1} +  t \imath)^{-1}$ in the statement of Theorem~\ref{thm.cons.transmission} and
$F = (f, g, h) = (f, 0, 0)$ with $F + \imath u = (f + t \imath u, 0, 0)$. We then obtain 
that $||| (\maP^{-1} +  t \imath)^{-1} ||| \le t^{-1}$ and hence 
following parametric regularity result (which is formulated, for convenience, for $t = 1$.).

\begin{theorem}\label{thm.param.reg}
There exist $c_{k,\maU} > 0$ and $p_k, q_k \in \ZZ_+$, $k \in \NN$, with the following property.
Let $A$ locally constant, scalar be as in Theorem~\ref{thm.self-adjoint}.
Let $k \ge 1$ and $u \in V_1$ be such that $\maP^A u \in V_k^-$. Then $u \in V_k \subset \check H^{k+1}(U_0)$ and
\begin{equation*}
   \|u\|_{\check H^{k+1}} \le c_{k, \maU} \nnW{k-1}{a_{mm}^{-1}}^{p_k} \|A\|_{W^{k,\infty}}^{q_k}
   \big ( \|\maP^A u \|_{V_k^-} + \|u\|_{H^1} \big )\,.
\end{equation*}
The constants $c_{k, \maU}$, $p_k$, and $q_k$ do not depend on $A$, $u$, or $\maP^A u$.
The constant $c_{k, \maU}$ may also depend on the domains $\maU:=\{U_j \mid j=0, \ldots N\}$.
\end{theorem}

\end{document}